\documentclass{amsart}
\usepackage{amssymb}
 
\newcommand{\bfZ}{{\bf  Z}}

\newcommand{\cB}{{\cal B}}
\newcommand{\cZ}{{\cal Z}}

\newcommand{\cL}{{\cal L}}
\newcommand{\cP}{{\cal P}}
\newcommand{\cF}{{\cal F}}
\newcommand{\cG}{{\cal G}}
\newcommand{\cN}{{\cal N}}
\newcommand{\cI}{{\cal I}}

\newcommand{\sur}{\twoheadrightarrow}

\newcommand{\cal}{\mathcal}

\newcommand{\N}{\cN}

\newcommand{\bu}{\bullet}

\newcommand{\I}{\cI}

\newcommand{\F}{\cF}
\newcommand{\Ftil}{\tilde{\cF}}
\newcommand{\G}{\cG}
\newcommand{\Gtil}{\tilde{\cG}}
\newcommand{\ftil}{\tilde f}
\newcommand{\jtil}{\tilde j}
\newcommand{\tilF}{\tilde{\cF}}

\newcommand{\Dpp}{{D^{p,\geq 0}(X)}}
\newcommand{\Dpm}{{D^{p,\leq 0}(X)}}

\newcommand{\cupl}{{\bigcup\limits}}

\newcommand{\pbar}{{\overline p}}

\newcommand{\Zet}{{\mathbb{Z}}}
\newcommand{\D}{{\mathbb{D}}}
\newcommand{\DC}{{\mathbb{DC}}}

\newcommand{\Ztop}{Z^{top}}

\newcommand{\Xtop}{X^{top}}
\newcommand{\tp}{^{top}} 
\newcommand{\tpl}{^{top}_{lc}} 
\newcommand{\ip}{{\mathbf i}} 
\newcommand{\md}{{}} 

\newcommand{\uHom}{{\underline {Hom}}}
\newcommand{\uExt}{{\underline {Ext}}}

\newcommand{\cO}{{\cal O}}

\newcommand{\coh}{\mathop{\sf Coh}} 
\newcommand{\qcoh}{\mathop{\sf QCoh}}
\newcommand{\ob}{\mathop{\mathrm{Ob}}}
\newcommand{\cone}{\mathop{\mathrm{Cone}}}
\newcommand{\codim}{\mathop{\mathrm{codim}}}

\newcommand{\supp}{\mathop{\mathrm{supp}}}
\newcommand{\spec}{\mathop{\mathrm{Spec}}}
\newcommand{\Char}{\mathop{\mathrm{char}}}
\newcommand{\Hom}{\mathop{\mathrm{Hom}}\nolimits}
\newcommand{\Ext}{\mathop{\mathrm{Ext}}\nolimits}

\newcommand{\A}{{\sf A}} 
\newcommand{\B}{{\sf B}} 

\newcommand{\sG}{{\mathcal{G}}}		
\newcommand{\sU}{{\mathcal{U}}}     
\newcommand{\sV}{{\mathcal{V}}}     
\newcommand{\sX}{{\mathcal{X}}}		
\newcommand{\sY}{{\mathcal{Y}}}		
\newcommand{\sZ}{{\mathcal{Z}}}		
\newcommand{\kk}{{k}}				
\newcommand{\kbar}{\overline\kk}	

\newcommand{\imbed}{\hookrightarrow}

\numberwithin{equation}{section}
\newtheorem{Thm}{Theorem}
\newtheorem{Cor}[Thm]{Corollary}
\newtheorem{Lem}[Thm]{Lemma}
\newtheorem{Prop}[Thm]{Proposition}
\newtheorem{Claim}[Thm]{Claim}

\theoremstyle{definition}
\newtheorem{Def}[Thm]{Definition}

\theoremstyle{remark}
\newtheorem{Rem}[Thm]{Remark}
\newtheorem*{Rem*}{Remark}
\newtheorem{Ex}[Thm]{Example}

\numberwithin{Thm}{section}

\begin{document}

\title[]{Perverse Coherent Sheaves}
\author{Dmitry Arinkin}
\address{Department of Mathematics, University of North Carolina at Chapel Hill,
NC, USA}
\email{arinkin@email.unc.edu} 
\author{Roman Bezrukavnikov}
\address{Department of Mathematics, Massachusetts Institute of Technology, MA, USA}
\email{bezrukav@math.mit.edu}
\subjclass{Primary 18F20; Secondary 14A20, 14F05}
\keywords{Coherent perverse sheaves, coherent IC sheaves}

\begin{abstract}
This note 
introduces an
 analogue of perverse $t$-structure \cite{BBD} on the derived category
of coherent sheaves on an algebraic stack (subject to some mild technical
conditions).
Under additional assumptions  construction of  
coherent ``intersection cohomology''
sheaves is given.  Those latter assumptions are rather restrictive but hold
in some  examples of interest in representation theory.


Similar results were obtained by Deligne (unpublished), Gabber \cite{Ga}
and Kashiwara \cite{Ka}.
\end{abstract}

\maketitle

\centerline{\em To Pierre Deligne with admiration.}
\section{Introduction} 
Let $X$ be a reasonable stratified topological space; or let
$X$ be a reasonable scheme, stratified by locally closed subschemes.
Let $D$ be the full subcategory in, respectively, derived category
of sheaves on $X$, or in  the derived category of \'etale sheaves  
on $X$,  consisting of complexes smooth along the stratification.

For an integer-valued function $p$ (perversity) on the set of strata
Beilinson, Bernstein, and Deligne \cite{BBD}
defined a  $t$-structure on the category $D$; the objects of corresponding
abelian category (core of the $t$-structure) are called perverse sheaves.

The question addressed in this note is whether an analogous construction
can be carried out for the derived category of coherent sheaves on a reasonable
scheme. Surprisingly, the answer is positive (with some modifications),
easy, and not widely known (although it was known to Deligne for a long time, 
see \cite{D}). 

Let us summarize the difference
between  the coherent case considered here, and
the constructible
case treated in \cite{BBD}.

First, in the coherent case
we can not work with  complexes ``smooth'' along a 
given stratification, because the corresponding subcategory in the derived category is not
a full triangulated subcategory. (If $f$ is a function whose divisor intersects
the open stratum, then the cone of the morphism $\cO \overset{f}{\to} \cO$
has singularity on the open stratum). This forces us to define perversity as a
function on the set of generic points of all 
irreducible subschemes, i.e. on the
topological space of a scheme. 

The second, more essential
 difference is that  in the derived category of coherent sheaves
the functor $j^*$ of  pull-back  under an open embedding $j$
does not have adjoint functors. 
Recall that in  constructible situation
the right adjoint to $j^*$ is the functor  $j_*$ of 
  direct image, and the left adjoint is the functor  $j_!$
of extension by zero.
In coherent set-up the functor $j_*$ is defined in the larger
 category of quasi-coherent sheaves (Ind-coherent sheaves),
 while $j_!$ is defined in the Grothendieck dual
 category (consisting 
of Pro-coherent sheaves) introduced in Deligne's appendix to \cite{H}.

It turns out, however, that in the proof 
 of the existence of perverse $t$-structure one can use
instead of the object  $j_!(\F)$ (where $j:U\imbed X$ is an open embedding)
any extension $\Ftil$ of $\F$ to $X$ such that
the restriction  of $\tilF$ to $X-U$ has no cohomology above certain degree
(depending on the perversity function).
 If the perversity function is
{\it monotone}  (see Definition \ref{mon} below) it is very easy to construct
such $\Ftil$. Applying the Grothendieck-Serre duality
to this construction, we get a substitute for $j_*(\F)$, which exists
if the perversity function is comonotone.
 Otherwise the proof is parallel to that in \cite{BBD}.

Thus the
 $t$-structure is constructed not for an
 arbitrary perversity function, but only for a monotone
and comonotone one. (In the topological situation one also needs this
condition to get a $t$-structure on the whole derived category
of constructible sheaves rather than on the category corresponding to a fixed
stratification.)

In \cite{D} Deligne used the
Grothendieck's Finiteness Theorem (\cite{SGA2}, VIII.2.1) to show that
 the formulas for 
$\tau ^p_{\leq 0}$, $\tau^p_{\geq 0}$ of \cite{BBD}, \emph{a priori} making sense
in a larger category containing $D^b(\coh)$, give in fact objects of $D^b(\coh)$,
provided the perversity function is monotone and comonotone
 (see also Remark \ref{Grfin}).

\medskip

The results on the existence of a ``perverse'' $t$-structure carry
over to the case of $G$-equivariant coherent sheaves, where $G$ is a
(reasonable) algebraic group acting on a (reasonable) scheme $X$. 
In this case  perversity $p(x )$ is assigned only to orbits
of $G$ (including ``generic orbits''), since
an equivariant sheaf is automatically smooth along the orbits.
 
More generally, the perverse $t$-structure can be constructed for
coherent sheaves on a (reasonable) algebraic stack $\sX$, given a
perversity defined on points of $\sX$. The case of $G$-equivariant 
coherent sheaves on $X$ corresponds to working with the stack
$\sX=X/G$. 

The general formalism in all three situations: sheaves on a scheme,
equivariant sheaves on a scheme, and sheaves on a stack --- is very similar,
to the extent that we found it easier to work with algebraic stacks
and treat the other situations as special cases. However, one construction
does not apply to (non-equivariant) sheaves on a scheme. 
Namely, the definition of the minimal (Goresky-MacPherson, or IC)
 extension functor $j_{!*}$ requires a {\it strictly}
monotone and comonotone perversity. Such perversity exists only if the dimensions
of adjacent points differ by at least two, which excludes schemes (other than finite schemes).
On the other hand, in equivariant settings, it is possible that the dimensions of adjacent
orbits differ by at least two, and a strictly monotone and comonotone perversity exists.

If the perversity is strictly monotone and comonotone, an analogue
of the usual description of irreducible perverse
sheaves as minimal extensions of local systems is valid,
and the core of the perverse $t$-structure is Artinian and Noetherian (in contrast with
the core of the standard $t$-structure). 
Some examples of this situation are given at the end of the paper.

\medskip

A version of the main result in a restricted generality appeared
in the preprint \cite{prpr} by the second author. Later
related constructions were published by Gabber  \cite{Ga} and 
Kashiwara \cite{Ka}.
A similar result was known to Deligne 
 \cite{D} 
long before the date of \cite{prpr}.

\subsection*{Organization} 
In Section~\ref{secprel}, we prove some basic properties of coherent sheaves, and  
study dualizing complexes on stacks.

Section~\ref{s1} contains the definition of the perverse $t$-structure on a stack. The main
result of this section (Theorem~\ref{t_str}) verifies axioms of a $t$-structure.

In Section~\ref{IC}, we define the minimal extension functor (Theorem~\ref{JGM}). We then use it
to study irreducible perverse coherent sheaves (Proposition~\ref{irre}), and prove that the 
category of perverse coherent sheaves is Artinian and Noetherian (Corollary~\ref{Art}).  
As we already mentioned, these results require additional assumptions; in particular,
all results are empty in the case of (non-equivariant) coherent sheaves on a scheme.

\subsection*{Acknowledgements} 
We  are much obliged to  Pierre Deligne for valuable explanations,
comments on the text and  a
kind permission to use his unpublished results.

We thank 
 Alexander Beilinson, Victor Ginzburg,
and Dmitry Panyushev for discussions and references.

This paper grew out of the preprint
\cite{prpr} which was started in the spring of 1999 when the second author was a
member at the Institute for Advanced Study. He thanks IAS for excellent work
conditions and Leonid Positselski for his
participation in the  early stages of the work on \cite{prpr}.

 The first author is
a Sloan Research Fellow, and he
 is grateful to the Alfred P.~Sloan Foundation for the 
support. The work of the second author was partly supported 
by DARPA grant  HR0011-04-1-0031
and NSF grant DMS-0625234.

\section{Preliminaries}\label{secprel}
In this section we collect some results needed in the exposition.

\subsection{Quasi-coherent sheaves on stacks}
Let $\sX$ be an algebraic stack (an Artin stack). Suppose $\sX$ is Noetherian, and in particular, quasi-compact. 
Assume also that $\sX$ is semi-separated: that is, the diagonal morphism $\sX\to\sX\times\sX$ is affine.
Let us fix a presentation of $\sX$, that is, a surjective smooth morphism
$\pi:X\to\sX$, where $X$ is an algebraic space. 

\begin{Lem}\label{semiseparated} 
Let $\sX$ be an algebraic stack. The following conditions are equivalent:
\begin{enumerate}
\item\label{semiseparated:1} $\sX$ is quasi-compact and semi-separated.

\item\label{semiseparated:2} 
$\sX$ admits a presentation $\pi:X\to\sX$ such that $X$ is an affine scheme and 
$\pi$ is an affine morphism. 

\item\label{semiseparated:3} 
$\sX$ admits a presentation $\pi:X\to\sX$ such that $X$ is a quasi-compact semi-separated algebraic space and $\pi$ is an affine morphism.
\end{enumerate}
\end{Lem}
\begin{proof}
\eqref{semiseparated:1}$\Rightarrow$\eqref{semiseparated:2}. Let $\pi:X\to\sX$ be a presentation
of $\sX$. Note that $X$ is quasi-compact, so passing to its \'etale cover, we may assume that
$X$ is an affine scheme. But then
$$X\times_\sX X=(X\times X)\times_{(\sX\times\sX)}\sX$$
is an affine scheme; therefore, $\pi$ is an affine morphism.

\eqref{semiseparated:2}$\Rightarrow$\eqref{semiseparated:3} is obvious.

\eqref{semiseparated:3}$\Rightarrow$\eqref{semiseparated:1}. Clearly, $\sX$ is quasi-compact; let us
prove it is semi-separated. The morphism
$$X\times_\sX X\to X$$ is affine, because it is obtained from $\pi$ by a base change. Since $X$ is
semi-separated, so is $X\times_\sX X$. Therefore, the composition
$$X\times_\sX X\to(X\times_\sX X)\times(X\times_\sX X)\to X\times X$$
is affine. It remains to notice that the composition is obtained from the diagonal
$\sX\to\sX\times\sX$ by a base change.
\end{proof}

\begin{Rem}
\label{Gorenstein}
To simplify the exposition, we only consider presentations $\pi:\sX\to X$ where
$X$ is a scheme from now on. This assumption allows us to avoid a separate discussion of 
perverse coherent sheaves on algebraic spaces.

Also, we do not need smoothness of presentations: it is enough to assume that 
$\pi:X\to\sX$ is a faithfully flat Gorenstein morphism of finite type. Recall that a flat morphism of finite type $\pi:X\to Y$ between locally Noetherian schemes 
is \emph{Gorenstein} if its fibers are Gorenstein schemes; equivalently, $\pi^!\cO_Y$ is an invertible sheaf (concentrated in a single cohomological dimension that need not be zero), see \cite[Exercise~V.9.7]{H}. The class of Gorenstein morphisms is local in smooth topology; this allows
to define Gorenstein morphisms between locally Noetherian stacks.

Let us agree that a \emph{presentation}
of a stack is a representable faithfully flat Gorenstein morphism of finite type $\pi:X\to\sX$,
where $X$ is a scheme. 
\end{Rem} 

Denote by $\coh(\sX)\subset\qcoh(\sX)$ the categories
of coherent and quasi-coherent sheaves on $\sX$, respectively. The presentation $\pi:X\to\sX$
defines a simplicial algebraic space
$X_\bullet$ (the \emph{coskeleton of $\pi$}): $X_i$
is the fiber product of $i+1$ copies of $X$ over $\sX$ ($i\ge 0$). We can interpret
quasi-coherent sheaves on $\sX$ as cartesian quasi-coherent sheaves on $X_\bullet$.

\begin{Ex} \label{equivariant}
An important example is the quotient stack $\sX=X/G$. Let us assume that $X$
is a semi-separated Noetherian scheme and
 $G$ is a flat finitely presented affine group scheme
(over some base scheme) acting on $X$. Then $\sX$ is an algebraic stack by the Artin criterion (see Theorem~10.1, Corollary~10.6 in \cite{La}); 
it is Noetherian and semi-separated. The natural morphism $X\to\sX$ is a 
presentation of $X$ if $G$ is smooth (or at least Gorenstein, 
since we consider presentations in the sense of Remark~\ref{Gorenstein}).
 In this example, quasi-coherent sheaves on $\sX$ are simply $G$-equivariant 
quasi-coherent sheaves on $X$.
\end{Ex}

\begin{Rem*}
Example~\ref{equivariant} demonstrates our main reason for working with Gorenstein presentations,
rather than smooth presentations as in \cite{La}. Namely, there are interesting group schemes $G$ that are Gorenstein, but not smooth:
for example, $G$ could be a non-reduced group scheme over a field of non-zero characteristic. By our definition, the morphism $X\to\sX$
is still a presentation of $\sX=X/G$ for such $G$.
\end{Rem*}

Denote by $D_{qcoh}(\sX)$ the derived category of quasi-coherent sheaves on $\sX$. 
Recall its definition (\cite{La}).

\begin{Def} \label{derived}
Consider the lisse-\'etale topology on $\sX$. (The underlying category is the category of 
morphisms $u:U\to\sX$, where $U$ is an algebraic space smooth over $\sX$; coverings are surjective \'etale morphisms.)
This site is equipped with the sheaf of rings $\cO_\sX$. Let $D(\cO_\sX)$ be the derived category of complexes of $\cO_\sX$-modules. 
Then $D_{qcoh}(\sX)\subset D(\cO_\sX)$ is the full subcategory formed by complexes whose cohomology sheaves are quasi-coherent, and
$D_{coh}(\sX)\subset D_{qcoh}(\sX)$ is the full subcategory of complexes with coherent cohomology. The notation $D^+_{coh}(\sX)$, $D^b_{qcoh}(\sX)$, and so on
is self-explanatory.
\end{Def}

$D_{qcoh}(\sX)$ can be described more explicitly in terms of the presentation 
$\pi:X\to\sX$. Let $X_\bu$ be the corresponding simplicial algebraic space. (More generally,
one can consider flat hypercovers $X_\bu\to\sX$.)

Consider the derived category of quasi-coherent sheaves on $X_\bu$ (cf. \cite[Appendix B]{BL}). We denote by
$D(\cO_{X_\bullet})$ the derived category of complexes of simplicial $\cO$-modules on
$X_\bullet$. The natural functor from $\qcoh(\sX)$ to the category of simplicial $\cO_{X_\bullet}$-modules is fully faithful. It provides an equivalence between $\qcoh(\sX)$ and the 
category of cartesian quasi-coherent sheaves on $X_\bu$. 
Define $D_{qcoh,cart}(X_\bullet)\subset D(\cO_{X_\bullet})$ to be the full subcategory of complexes
whose cohomology objects are cartesian quasi-coherent sheaves.

\begin{Claim} \label{topology}
\begin{enumerate}
\item\label{topology1} In Definition~\ref{derived}, the lisse-\'etale topology can be replaced by 
the smooth topology or the fppf topology,  or a `hybrid topology' such as
`fppf-Zariski'. If $\sX$ is a Deligne-Mumford stack (resp. scheme), one can work with
the \'etale topology (resp. Zariski topology) on $\sX$. In all these cases, the corresponding
versions of $D_{qcoh}(\sX)$ are naturally equivalent.
\item\label{topology2} Similarly, in the definition of $D_{qcoh,cart}(X_\bullet)$, one can use on $X_i$'s any
of the usual topologies: smooth, flat, \'etale, or Zariski (assuming $X_i$'s are schemes).
The corresponding derived categories are naturally equivalent. 
\end{enumerate}
\end{Claim}

\begin{Claim} The natural functor $$D_{qcoh}(\sX)\to D_{qcoh,cart}(X_\bu)$$ is an equivalence
of categories. 
\label{simplicial}
\end{Claim}

\begin{Claim} $D_{qcoh}^+(\sX)$ is identified with the bounded below derived category of the abelian
category $\qcoh(\sX)$. 
\label{derivedqcoh}
\end{Claim}

Claims~\ref{topology} and \ref{simplicial} are well known  for the bounded below derived category $D^+_{qcoh}(\sX)$. In this case, they fit into
the framework of cohomological descent \cite{descent} (see also \cite[Chapter 5]{D2}); see Theorem~13.5.5 in \cite{La} for a proof of Claim~\ref{simplicial} under this assumption. Laszlo and Olsson 
extended cohomological descent to unbounded derived categories in \cite{LO}; the claims easily follow.

\begin{proof}[Proof of Claims~\ref{topology} and \ref{simplicial}] 
Claim~\ref{simplicial} immediately follows from Theorem~2.2.3 of \cite{LO} (see \cite[Example~2.2.5]{LO}). 
To prove Claim~\ref{topology}\eqref{topology1}, 
we first note that there are natural functors relating different 
versions of the derived category (the functors are the direct and inverse images under the continuous maps between the corresponding sites).
We need to verify that the functors are equivalences. This also follows from \cite[Theorem~2.2.3]{LO}.
(Technically, \cite[Theorem~2.3.3]{LO} only considers maps from a simplicial topos to a topos, but
any site can be viewed as a simplicial site.) 
Same argument applies to sites associated with different topologies on $X_\bu$, proving Claim~\ref{topology}\eqref{topology2}.
\end{proof}

Let us now prove Claim~\ref{derivedqcoh}. Fix a presentation $\pi:X\to\sX$ such that $X$ is an affine scheme and
$\pi$ is an affine morphism; such presentation exists by Lemma~\ref{semiseparated}. Let
$X_\bu$ be the corresponding simplicial scheme. Consider the derived categories
$D_{qcoh}(\sX)$ and $D_{qcoh,cart}(X_\bu)$.

Claim~\ref{simplicial} provides an equivalence between the categories. The equivalence is given
by functors $\pi_{\bu *}$ and $\pi_\bu^*$, where $\pi_\bu$ stands for the morphism $X_\bu\to\sX$.
(We often use the same notation for a functor on an abelian category and its derived functor on the derived category.) 

Now consider the derived category $D_{cart}(\qcoh(X_\bu))$. Its objects are
complexes of quasi-coherent sheaves on $X_\bu$ whose cohomology sheaves are cartesian. 

\begin{Lem}\label{DCom} The natural functor $D_{cart}(\qcoh(X_\bu))\to D_{cart,qcoh}(X_\bu)$ is an equivalence 
of categories.
\end{Lem}
\begin{proof} Once again, we derive the statement from \cite[Theorem~2.2.3]{LO}. Let $Y_\bu$ be the final object in
the category of simplicial topological spaces: $Y_i$ is a singleton set for all $i$. Equip $Y_i$ with a 
sheaf of rings $\cO_{Y_i}$ by setting
$$H^0(Y_i,\cO_{Y_i})=H^0(X_i,\cO_{X_i})\quad(i\ge 0).$$
Clearly, this gives rise to a simplicial sheaf of rings $\cO_{Y_\bu}$ on $Y_\bu$ and a morphism of simplicial ringed sites 
$\Gamma_\bu:X_\bu\to Y_\bu$. (Of course, a simplicial sheaf of rings on $Y_\bu$ is nothing but a cosimplicial ring.)

The inverse image $\Gamma_\bu^*$ identifies the category of $\cO_{Y_\bu}$-modules and the category of quasi-coherent $\cO_{X_\bu}$-modules.
By \cite[Theorem~2.2.3]{LO}, the functors $\Gamma_{\bu *}$ and $\Gamma_\bu^*$ induce mutually inverse equivalences 
between $D_{cart,qcoh}(X_\bu)$ and $D_{cart}(Y_\bu)$ 
(the derived category of complexes of $\cO_{Y_\bu}$-modules 
whose cohomology objects are cartesian).
\end{proof}

\begin{Rem*} Lemma~\ref{DCom} is a simplicial version of \cite[Theorem~5.1]{BN}.
\end{Rem*}

By Lemma~\ref{DCom}, we can restate Claim~\ref{derivedqcoh} using the category $D^+_{cart}(\qcoh(X_\bu))$
in place of $D^+_{cart,qcoh}(X_\bu)$. The advantage of working with $D^+_{cart}(\qcoh(X_\bu))$ is that the functor $\pi_{\bu*}$ can be 
described more explicitly. Namely, for $\F\in\qcoh(X_\bu)$, consider its \emph{\v Cech complex} 
$$
{\check\pi_{\bu*}}\F=(\pi_{0*}\F_0\to\pi_{1*}\F_1\to\dots).
$$
Here we denote by $\F_i$ the component of $\F$ concentrated on $X_i$; $\pi_i:X_i\to\sX$ is the projection. The definition immediately extends to complexes of quasi-coherent sheaves on $X_\bu$.
This yields a functor $$\check\pi_{\bu *}:D^+_{cart}(\qcoh(X_\bu))\to D^+(\qcoh(\sX))$$
such that the composition
$$ D^+_{cart}(\qcoh(X_\bu))\to D^+(\qcoh(\sX))\to D^+_{qcoh}(\sX)$$ is naturally isomorphic
$\pi_{\bu *}$.

\begin{proof}[Proof of Claim~\ref{derivedqcoh}]
Let us show that $\check\pi_{\bu *}$ is an equivalence, with inverse equivalence given by $\pi^*_\bu$.
An isomorphism between $\pi_\bu^*\circ\check\pi_{\bu *}=\pi_\bu^*\circ\pi_{\bu *}$ and the identity
functor is given by Claim~\ref{simplicial}. On the other hand, for any $\F\in\qcoh(\sX)$, we have
a natural map $$\F\to\check\pi_{\bu*}(\pi_\bu^*\F).$$
It leads to a functorial morphism from the identity functor to $\check\pi_{\bu *}\circ\pi_\bu^*$.
Clearly, it is an isomorphism of functors.
\end{proof}

\begin{Rem*} Let us sketch another proof of Claim~\ref{derivedqcoh} for bounded derived categories. 
We need to show that the natural functor
$$
D^b(\qcoh(\sX))\to D^b_{qcoh}(\sX)
$$
is an equivalence. It is enough to verify that it is fully faithful; then essential surjectivity
follows because every $\F\in D^b_{qcoh}(\sX)$ is obtained from quasi-coherent sheaves by taking cones.
For the same reason, we only need to check that the natural morphism of functors
\begin{equation}
\label{eqcomp}
\Ext^i_{\qcoh(\sX)}(\F,\G)\to\Ext^i_{\cO_\sX}(\F,\G)\quad(\F,\G\in\qcoh(\sX))
\end{equation}
is an isomorphism for all $i$. This is trivial for $i=0$. For $i>0$, it suffices to check that
both sides of \eqref{eqcomp} are effaceable as functors of $\G\in\qcoh(\sX)$.

Let $\pi:X\to\sX$ be an affine presentation. For any $\G\in\qcoh(\sX)$, pick an embedding 
$\pi^*\G\hookrightarrow\I$ into injective $\I\in\qcoh(X)$. Then $\G\to\pi_*\I$ is also an embedding.
It remains to notice that
$$\Ext^i_{\qcoh(\sX)}(\F,\pi_*\I)=\Ext^i_{\cO_\sX}(\F,\pi_*\I)=0\quad(\F\in\qcoh(\sX),i>0)$$
by adjunction.
\end{Rem*}

\begin{Rem*} B\"okstedt and Neeman proved that for a quasi-compact separated scheme $X$, the unbounded derived categories $D_{qcoh}(X)$ and $D(\qcoh(X))$ are equivalent (\cite[Corollary~5.5]{BN}). It is possible that Claim~\ref{derivedqcoh} also holds for unbounded derived categories.

Note however that a finite affine cover of a quasi-compact semi-separated scheme $X$ gives rise to a `bounded' affine semi-simplicial scheme
$X_\bu$ (that is, $X_i=\emptyset$ for $i\gg0$). The corresponding \v Cech complex is bounded, so
in this case the proof of Claim~\ref{derivedqcoh} extends to unbounded derived categories (the argument is then almost identical to that of B\"okstedt and Neeman, 
see \cite[Section~6.7]{BN}). 
In the case of stacks, such simplification is no longer available.
\end{Rem*}

\subsection{Coherent sheaves on stacks} Consider now the derived category $D_{coh}(\sX)$.
\begin{Lem}[{\cite[Proposition~15.4]{La}}] \label{union}
Any $\F\in\qcoh(\sX)$ is the union of its coherent subsheaves. \qed
\end{Lem}

\begin{Cor} \label{odinhren}           
Let $\F^\bu$ be a bounded above complex of quasi-coherent sheaves on $\sX$ whose
cohomology is coherent. Then $\F^\bu$ contains a coherent subcomplex $\F_c^\bu\subset\F^\bu$
quasi-isomorphic to $\F^\bu$. Moreover, for any coherent subcomplex $\G^\bu\subset\F^\bu$, 
we can choose such $\F_c^\bu$ so that $\F_c^\bu\supset\G^\bu$.
\end{Cor}

\begin{proof} The argument is almost standard (similar to \cite[Proposition~I.4.8]{H}, for instance). We can construct 
subsheaves $\F^i_c\subset\F^i$ by descending induction in $i$. We require that $\F^i_c$ has
the following properties:
\begin{gather*}
d(\F^i_c)=\F^{i+1}_c\cap \cB^{i+1}\\
(\F^i_c\cap \cZ^i)\sur \cZ^i/\cB^i=H^i(\F^\bu)\\
\F_c^i\supset \G^i.
\end{gather*}
Here $\cZ^i, \cB ^i\subset \F^i$ denote, respectively, the kernel and the image of the differential.

Such $\F^i_c$ exists by Lemma~\ref{union}. Indeed, $d^{-1}(\F^{i+1}_c)$ is a union of its coherent subsheaves;
since $$d(d^{-1}(\F^{i+1}_c))=\F^{i+1}_c\cap \cB^{i+1}$$
is coherent (and therefore Noetherian), there exists a coherent subsheaf $\F^i_1\subset d^{-1}(\F^{i+1}_c)$
such that $d(\F^i_1)=\F^{i+1}_c\cap \cB^{i+1}$. Similarly, there exists a coherent subsheaf $\F^i_2\subset\cZ^i$
whose projection to $H^i(\F^\bu)$ is surjective. Take $\F_c^i=\G^i+\F^i_1+\F^i_2$.
\end{proof}

\begin{Cor} 
The category $D^b_{coh}(\sX)$ is naturally equivalent to the derived category $D^b(\coh(\sX))$
of complexes of coherent sheaves.
\end{Cor}
\begin{proof} Consider the category $D^b_{coh}(\qcoh(\sX))$ (the full subcategory of
$D^b(\qcoh(\sX))$ formed by complexes with coherent cohomology). By Claim~\ref{derivedqcoh}, it is
equivalent to $D^b_{coh}(\sX)$. Therefore, we need to show that the natural functor
$$D^b(\coh(\sX))\to D^b_{coh}(\qcoh(\sX))$$
is an equivalence. Essential surjectivity follows from Corollary~\ref{odinhren}. Let us show that the functor is faithful.
Suppose $$f:\F^\bu\to\G^\bu$$ is a morphism between bounded complexes of coherent sheaves whose image in $D^b(\qcoh(\sX))$ vanishes.
Then there exists a quasi-isomorphism $$\iota:\F_1^\bu\to\F^\bu$$
such that $f \iota$ is homotopic to zero. Here $\F_1^\bu$ is a bounded
complex of quasi-coherent sheaves. Using Corollary~\ref{odinhren}, we replace $\F_1^\bu$ by a quasi-isomorphic coherent subcomplex. This shows
that the image of $f$ in $D^b(\coh(\sX))$ also vanishes.

Now let us prove the functor is full. Given bounded complexes of coherent sheaves $\F^\bu,\G^\bu$, a morphism between 
their images in $D^b(\qcoh(\sX))$ is of the form 
$f \iota^{-1}$, where $f:\F_1^\bu\to\G^\bu$ is a morphism between complexes of quasi-coherent 
sheaves, and $\iota:\F_1^\bu\to\F^\bu$ is a quasi-isomorphism. 
By Corollary~\ref{odinhren}, we can replace $\F_1^\bu$ by a quasi-isomorphic coherent
subcomplex. Then $f\iota^{-1}$ is defined in $D^b(\coh(\sX))$.
\end{proof} 

We also need statements relating coherent sheaves on $\sX$ to coherent sheaves
on its substacks.

\begin{Lem}\label{exte} Let
$\sU\subset\sX$
be an open substack. 

\begin{enumerate}
\item\label{exte:a} (cf. \cite[Corollary 15.5]{La}) Any $\F\in D^b_{coh}(\sU)$ extends to 
$\Ftil\in D^b_{coh}(\sX)$ such that $\Ftil|_\sU\simeq\F$.

\item\label{exte:b}
For any $\F,\G\in D^b_{coh}(\sU)$ and a morphism $f:\F\to \G$, we can choose extensions
$\Ftil, \Gtil\in D^b_{coh}(\sX)$ and a morphism
$\ftil:\Ftil\to \Gtil$ such that $\ftil|_U \cong f$. 

\item\label{exte:c}
For any $\Ftil_1, \Ftil_2\in D^b_{coh}(\sX)$ 
and an isomorphism $f:\Ftil_1|_\sU\cong \Ftil_2|_\sU$, there exists $\Ftil\in D^b_{coh}(\sX)$ 
and morphisms $f_1:\Ftil_1\to \Ftil$, $f_2:\Ftil_2\to \Ftil$ such that 
$f_1|_\sU$ and $f_2|_\sU$ are isomorphisms and $f=(f_1|_\sU)\circ(f_2|_\sU)^{-1}$.
\end{enumerate}
\end{Lem}

\begin{proof} The argument is parallel to the proof of Corollary~\ref{odinhren}.  
Indeed, $\F\in D^b_{coh}(\sU)$ can be represented by a bounded complex of coherent 
$\cO_\sU$-modules $\F^\bu$. The complex extends to a complex of quasi-coherent $\cO_\sX$-modules $j_*\F^\bu$. Here $j:\sU\imbed\sX$ is the embedding. Replace
$j_*\F^\bu$ by a coherent subcomplex $\F_c^\bu$ such that the restriction 
$\F_c^\bu|_\sU \imbed j_*\F^\bu|_\sU=\F^\bu$ is a quasi-isomorphism (arguing as in
Corollary~\ref{odinhren}).
This proves \eqref{exte:a}.

Given a morphism $f:\F\to\G$ as in part~\eqref{exte:b}, choose representatives
$\F^\bu,\G^\bu$ in such a way that $f$ can be represented by a chain map
$\F^\bu\to\G^\bu$. Choose a coherent subcomplex 
$\F_c^\bu\subset j_*\F^\bu$ first, and then choose a coherent subcomplex
$\G_c^\bu\subset j_*\G^\bu$ containing $f(\F_c^\bu)$. 

For part \eqref{exte:c}, represent $\Ftil_i$ by a bounded complex of coherent
$\cO_\sX$-modules $\Ftil_i^\bu$ ($i=1,2$). The map $f$ then can be written as
$g_1\circ g_2^{-1}$ for a bounded complex of coherent $\cO_\sU$-modules $\F^\bu$ and quasi-isomorphisms $g_i:\Ftil_i^\bu|_\sU\to\F^\bu$. Finally, extend $\F^\bu$ to a complex
$\Ftil^\bu$ of coherent $\cO_\sX$-modules such that $g_i$ extends to $f_i:\Ftil_i^\bu\to\Ftil^\bu$ for $i=1,2$ by the argument used to prove part~\eqref{exte:b}. 

\end{proof}

\begin{Rem*}
Lemma~\ref{exte} implies that $D^b_{coh}(\sU)$ is equivalent to the localization 
of $D^b_{coh}(\sX)$ 
with respect to the class of morphisms that become isomorphisms when restricted to $\sU$.
\end{Rem*}

For a stack $\sX$, we denote by $\sX\tp$ the set of points of $\sX$ equipped with the Zariski
topology. Recall (\cite[Chapter~5]{La}) that a \emph{point} of $\sX$ is an equivalence class
of pairs $(K,\xi)$, where $K$ is a field and $\xi:\spec(K)\to\sX$ is a morphism. Two such pairs
$(K,\xi)$, $(K',\xi')$ represent the same point if the fields $K$ and $K'$ have a common
extension $K''$ such that the pull-backs of $\xi$ and $\xi'$ to $\spec(K'')$ are isomorphic.
For any open substack $\sU$ of $\sX$, the points of $\sU$ form a subset of $\sX\tp$; we define the Zariski topology on $\sX\tp$ by declaring such subsets open.

\begin{Lem} \label{closedimbed}
Let $\sX\tp$ be the set of points of $\sX$ equipped with the Zariski topology. Given
closed $\bfZ\subset\sX\tp$ and $\F\in D^b_{coh}(\sX)$ such that $\supp\F\subset\bfZ$,
there exists a closed substack $i_{\sZ}:\sZ\hookrightarrow\sX$ with $\sZ\tp\subset\bfZ$ such that
$$\F\cong i_{Z*} (\F_Z)\text{ for some }\F_Z\in D^b_{coh}(\sZ).$$
\end{Lem}

\begin{proof}
Let $\qcoh_\bfZ(\sX)\subset\qcoh(\sX)$ be the full subcategory of sheaves supported by $\bfZ$.
It is easy to see that the $\qcoh_\bfZ(\sX)$ has enough objects that are injective in $\qcoh(\sX)$; therefore, $\F$ can be represented by a (bounded from below) complex of injective objects of $\qcoh_\bfZ(\sX)$.
Truncation gives a bounded complex whose terms need not be injective. 
Finally, by Corollary~\ref{odinhren} we see that $\F$ can be represented by a complex of coherent sheaves supported by $\bfZ$. This implies the statement.
\end{proof}

\subsection{Dualizing complexes on stacks} As before, $\sX$ is a Noetherian algebraic stack. Recall
(\cite[\S V.2]{H}) that the dualizing complex on a Noetherian scheme $X$ is an object
$\DC_X\in D^b_{coh}(X)$ of finite injective dimension (in $\qcoh(X)$) such that
every $\F\in D^b_{coh}(X)$ is $\DC_X$-reflexive. That is, the natural transformation
$$\F\to\uHom(\uHom(\F,\DC),\DC)\quad(\F\in D^b_{coh}(X))$$
is an isomorphism.

For a flat Gorenstein morphism $\pi:Y\to X$ of finite type, we denote by
$\omega_{Y/X}$ the relative dualizing sheaf, which is a line bundle on $Y$. 
Let us include the appropriate cohomological shift by $\dim(Y/X)$ 
in $\omega_{Y/X}$, so that 
\begin{equation}\label{pi!}
\pi^!(\F)=\pi^*(\F)\otimes\omega_{Y/X}\quad(\F\in D^b_{coh}(X)),
\end{equation}
as in \cite[Remark on pp. 143--144]{H}. 

\begin{Rem} Let us review the construction of $\omega_{Y/X}$ sketched in 
\cite[Section~7.4]{H}. As before, $\pi:Y\to X$ is a flat Gorenstein morphism of finite
type. It is easy to define $\pi^!$ if $\pi$ is embeddable (that is, $\pi$
is a composition of a closed embedding and a smooth morphism). If $\pi$ is embeddable,
we set $\omega_{Y/X}=\pi^!(\cO_X)$. The construction of $\omega_{Y/X}$ is local on $Y$, so we 
can define it for arbitrary (not necessarily embeddable) $\pi$ by passing to a cover of $Y$.
Indeed, every morphism of finite type is locally embeddable.
We can then use \eqref{pi!} to define $\pi^!$ for arbitrary $\pi$.

Note that the construction of $\omega_{Y/X}$ is also local in flat topology of $X$; this
allows us to define the relative dualizing sheaf and $\pi^!$ if $\pi$ is a  representable flat Gorenstein morphism of finite type between Noetherian algebraic stacks.
\end{Rem}

The notion of a dualizing complex is local with respect to flat Gorenstein covers:
\begin{Lem} \label{locdual}
Let $\pi:Y\to X$ be a faithfully flat Gorenstein morphism of finite type between Noetherian schemes. Then $\F\in D^b_{coh}(X)$ is a dualizing complex on $X$ if and only if
$\pi^*\F$ is a dualizing complex on $Y$. 
\end{Lem}
\begin{proof} The 'only if' direction is clear: if $\F$ is a dualizing complex on $X$, 
then $\pi^!\F$ is a dualizing complex on $Y$ 
(\cite{H}, Theorem~V.8.3 and the remark following Corollary~V.8.4). 
Since $\pi^*\F$ differs from $\pi^!\F$ by a twist by a line bundle, it is also dualizing. 

For the 'if' direction, we note that
$$\pi^*R\uHom(\G,\F)=R\uHom(\pi^*\G,\pi^*\F)\quad\text{for all }\G\in D^b_{coh}(X).$$ 
Therefore, $\F$ has finite local injective dimension:
for $i\gg0$, we have $\uExt^i(\G,\F)=0$ for all $\G\in\coh(X)$. Also, all $\G\in D^b_{coh}(X)$ 
are $\F$-reflexive. It remains to notice that $\F$ has finite injective dimension by
\cite[Proposition~II.7.20]{H}.
\end{proof}

\begin{Def}\label{defdual} Let $\sX$ be a Noetherian stack. A \emph{dualizing complex} on $\sX$
is an object $\DC_\sX\in D^b_{coh}(\sX)$ such that for a presentation $\pi:X\to\sX$,
$\pi^*\DC_\sX$ is a dualizing complex on $X$. 
\end{Def}

By Lemma~\ref{locdual}, Definition~\ref{defdual} does not depend on the presentation $\pi$.

\begin{Ex} Suppose the ground scheme $S$ is Noetherian and admits a dualizing complex. 
Let $G$ be a separated Gorenstein $S$-group scheme of finite type acting on a $S$-scheme $X$ of finite type. 
By Definition~\ref{defdual}, a dualizing complex $\DC_\sX$ on $\sX=X/G$ can be viewed as a $G$-equivariant dualizing complex on $X$: a $G$-equivariant complex is an equivariant dualizing complex if and only if it is a dualizing complex on $X$.
Proposition~\ref{dual_exist} shows that $\DC_\sX$ exists.
\end{Ex}

\begin{Rem*} $\DC_\sX$ has finite local injective dimension, and all
objects of $D^b_{coh}(\sX)$ are $\DC_\sX$-reflexive. 
On the other hand, the injective dimension of $\DC_\sX$ may be infinite. 

For instance, let $\sX$ be the classifying stack of a group $G$ over a field $\kk$. Then $\qcoh(\sX)$ is identified
with the category of representations of $G$, and we can take $\DC_\sX$ to be the trivial one-dimensional 
representation of $G$. 
Its injective dimension may be infinite. For example, this happens if $\Char(\kk)>0$ and 
$G$ is a cyclic group of order $\Char(\kk)$.
We are grateful to Leonid Positselski for pointing this out. 
\end{Rem*}

\begin{Prop}\label{dual_exist} Suppose $\sX$ is of finite type over a Noetherian scheme $S$ that admits a dualizing complex.
Then $\sX$ admits a dualizing complex. 
\end{Prop}

\begin{proof} Let $\pi:X\to\sX$ be a presentation of $\sX$. To simplify matters further,
we may assume that $f_X:X\to S$ is embeddable (by replacing $X$ with its Zariski open cover). 
Fix a dualizing complex $\DC_S\in D^b_{coh}(S)$. Then $f_X^!(\DC_S)$ is a dualizing complex on
$X$, and so is
$$\F=f_X^!(\DC_S)\otimes\omega_{X/\sX}^{-1}\in D^b_{coh}(X),$$ 
since it is obtained from $f_X^!(\DC_S)$ by a twist by a line bundle. Let us show that $\F$ 
descends to $\sX$. 

By \cite[Theorem~3.2.4]{BBD} (see also \cite[Section~3.2.5]{BBD}), 
an object of the derived category of sheaves on a site can be given locally 
provided negative local $\Ext$'s from the object to itself vanish. Applying it to the flat covering
$\pi:X\to\sX$, we see that it is enough to provide
an isomorphism 
\begin{equation}\label{descent}
\pi_2^*\sX\simeq\pi_1^*\sX\in D^b_{coh}(X\times_\sX X)
\end{equation} 
satisfying the associativity constraint
on $X\times_\sX X\times_\sX X$. Here $\pi_{1,2}:X\times_\sX X\to X$ are the projections.

However, $\pi_{1,2}$ are Gorenstein, so $\pi_{1,2}^!$ and $\pi_{1,2}^*$ are related by 
\eqref{pi!}. Also, the relative canonical bundle agrees with composition, so 
$$\pi_i^*\omega_{X/\sX}\otimes\omega_{X\times_\sX X/X}=\omega_{X\times_\sX X/\sX}\quad(i=1,2).$$
We therefore see that
$$\pi_i^*\F=\pi_i^!\F\otimes\omega_{X\times_\sX X/X}^{-1}=f_{X\times_\sX X}^!(\DC_S)\otimes\omega_{X\times_\sX X/\sX}^{-1}\quad(i=1,2),$$ 
where $f_{X\times_\sX X}:X\times_\sX X\to S$.
This provides isomorphism \eqref{descent}. 

The associativity constraint is verified in a similar way: we identify the pull-backs of $\F$
to $X\times_\sX X\times_\sX X$ (with respect to three different projections on $X$) with
$$f_{X\times_\sX X\times_\sX X}^!(\DC_S)\otimes\omega_{X\times_\sX X\times_\sX X/\sX}^{-1},$$ and check that pull-backs of \eqref{descent} (with respect to three different projections on
$X\times_\sX X$) become the identity maps under this identification. This reduces to
functorial properties of pull-back $f^!$, namely, to associativity of isomorphism $(fg)^!=g^!f^!$ under composition of three morphisms (Condition~VAR~1 of \cite[Theorem~III.8.7]{H}).
\end{proof}

\begin{Rem*} The dualizing complex on $\sX$ can be constructed more explicitly if $\sX\to S$ is embeddable: 
there exists a closed embedding $\iota:\sX\hookrightarrow\sY$ into a smooth $S$-stack $\sY$. 
In this case, we can take
$$\DC_\sX=\iota^!(f_{\sY}^*(\DC_S)\otimes\omega_{\sY/S})$$ for $f_\sY:\sY\to S$.
(A less canonical choice is $\DC_\sX=\iota^!f_{\sY}^*(\DC_S)$.)

In particular, suppose $X$ is a normal quasiprojective variety over an algebraically closed field, and $G$
is a connected linear algebraic group acting on $X$. Sumihiro's embedding theorem (\cite{Su}, see also \cite{KK}) claims that 
there exists an action of $G$ on ${\mathbb P}^{\mathbf n}$ and a $G$-equivariant embedding
$X\hookrightarrow{\mathbb P}^{\mathbf n}$.  Thus the stack $\sX=X/G$ admits a closed embedding
$\sX\hookrightarrow{\mathbb P}^{\mathbf n}/G$.
\end{Rem*}

\subsection{$*$-restriction and $!$-restriction}
We now summarize the properties of restriction to a (not necessarily) closed point.
To simplify the exposition, we only use the functors on schemes. Let $X$ be a Noetherian scheme. 

\begin{Rem*} We do not need to assume that $X$ is semi-separated. This assumption is only used
in Claim~\ref{derivedqcoh}, but it is well known that the claim holds for locally Noetherian 
schemes (\cite[Corollary~II.7.19]{H}).
\end{Rem*}

Recall that $X\tp$ is the underlying topological space of $X$ equipped with the Zariski topology. Given a point $x\in X\tp$, we write $\ip_x :\{x\}\imbed X\tp$ for the natural embedding.

Let $\F$ be a quasi-coherent sheaf of $\cO_X$-modules on $X\tp$. 
Then $\ip_x^*\F$ is its stalk at $x$. Thus $\ip_x^*$ is an exact functor from $\qcoh(X)$ to
the category of modules over the local ring $\cO_x=\ip_x^*\cO_X$. We keep the same notation for its 
derived functor
$$\ip_x^*:D_{qcoh}(X)\to D(\cO_x\md).$$ Here $D(\cO_x\md)$ is the derived category of $\cO_x$-modules.

Denote by $\Gamma_x(\F)\subset\ip_x^*\F$ the submodule of sections with support in 
$\overline{\{x\}}$ (as in Variation~8 of \cite[IV.1]{H}). Note that $\Gamma_x(\F)$ is
a torsion $\cO_x$-module, that is, any its element is annihilated by a power of the maximal 
ideal of $\cO_x$. Actually, $\Gamma_x(\F)$ can be defined as the maximal torsion submodule 
of $\ip_x^*\F$. The functor $\Gamma_x$ is left exact, and we denote its derived functor by
$$\ip_x^!:D^+_{qcoh}(X)\to D^+(\cO_x\md).$$ 

The functor $\ip_x^!$ can be described as follows. Set $\bfZ=\overline{\{x\}}\subset X\tp$, 
and let $\ip_{\bfZ}:\bfZ\to X\tp$ be the closed embedding of topological spaces. 
For $\F\in\qcoh(X)$, we can identify $\Gamma_x(\F)$ with the stalk at $x$ of the 
$!$-restriction $\ip_{\bfZ}^!(\F)$. 

\begin{Lem} $\ip_x^!$ has finite cohomological dimension.
\end{Lem}
\begin{proof} It is enough to show that $\ip_{\bfZ}^!$ has finite cohomological dimension for
$\bfZ=\overline{\{x\}}\subset X\tp$. We have a natural distinguished triangle 
$$\ip_{\bfZ*}\ip^!_\bfZ \F \to \F \to j_*j^*(\F)\to \ip_{\bfZ*}\ip^!_\bfZ \F[1]\qquad(\F\in D^+_{qcoh}(X))$$
in $D_{qcoh}(X)$. Here $j$ is the embedding of the open subscheme $X\tp-\bfZ$. It remains 
to notice that $j_*$ has finite cohomological dimension.
\end{proof}

\begin{Rem*} Note that $\ip_\bfZ^!\F$ is not quasi-coherent: it is a (complex of) $\ip_\bfZ^*\cO_X$-modules. On the other hand, $\ip_{\bfZ*}\ip_\bfZ^!\F$ is quasi-coherent. In the notation of 
\cite{H} (Section~IV.1, Variation~3),  $\ip_{\bfZ*}\ip_\bfZ^!\F=R\underline{\Gamma}_\bfZ(\F)$.
\end{Rem*}

Let us compare the topological restriction functors $\ip_x^*$ and $\ip_x^!$ with their
$\cO$-module counterparts. Let $Z\subset X$ be a locally closed subscheme; consider the embedding
$i_Z:Z\hookrightarrow X$. We then have two restriction functors
$$i_Z^*:D^-_{qcoh}(X)\to D^-_{qcoh}(Z)\quad\text{and}\quad i_Z^!:D^+_{qcoh}(X)\to D^+_{qcoh}(Z),$$
which are derived of a right exact and a left exact functor, repectively.

\begin{Ex} If $Z\subset X$ is open, $i_Z^*=i_Z^!$ is the restriction functor. If $Z\subset X$
is closed, 
$$i_Z^*(\F)=\F\overset{L}\otimes_{\cO_X}\cO_Z\quad\text{and}\quad i_Z^*(\F)=R\uHom_{\cO_X}(\cO_Z,\F).$$
The functors for general locally closed embedding $i_Z$ are obtained by composing these two 
special cases.
\end{Ex}

Note that the functors $i_Z^*$ and $i_Z^!$ preserve coherence. Also, they make
sense for locally closed embeddings of stacks. 

\begin{Lem}\label{equiv} Let    
  $Z\subset X$ be a locally closed 
 subscheme,   and
 $n$ be an integer. Let $x \in X\tp$ be a generic point of $Z$. Then

\begin{enumerate}
\item\label{equiv:a}  For $\F \in D^-_{coh}(X)$ we have $\ip_x ^*(\F)\in D^{\leq n}(\cO _x\md)$ if and only if
there  exists an open subscheme $Z^0\subset Z$, $Z^0\owns x$
 such that $i_{Z^0}^*(\F)\in D^{\leq n}_{coh}(Z_0)$;

\item\label{equiv:b}  For $\F \in D^+_{coh}(X)$ 
we have $\ip_x ^!(\F)\in D^{\geq n}(\cO_x\md)$ 
if and only if there
  exists  an open  subscheme $Z^0\subset Z$, $Z^0\owns x$
such that $i_{Z^0}^!(\F )\in D^{\geq n}_{coh} (Z_0)$.
\end{enumerate}
\end{Lem}

\begin{proof} Existence of an open subscheme $Z^0\subset Z$
as in \eqref{equiv:a} is equivalent to $$\ip_x^*i_Z^*(\F)\in D^{\leq n}(\cO_{Z,x}\md).$$
Indeed, if this  holds, 
we can let $Z^0$ be 
the complement in $Z$ to support of ${\cal H}^k (i_Z^*(\F))$, $k>n$;
the converse is obvious.

Let us rewrite $$\ip_x^*i_Z^*(\F)
=\ip_x^*(\F)\overset{L}{\otimes}_{\cO_x} \cO_{Z,x}.$$ Since the functor
of tensor product with $\cO_{Z,x}$ over $\cO_x$ is right exact, and kills no
finitely generated $\cO_x$-modules by the Nakayama Lemma, we see that 
the top cohomology of $\ip_x^*(\F)\overset{L}{\otimes}_{\cO_x} \cO_{Z,x}$
and of $\ip_x^*(\F)$ occur in the same degree. This proves \eqref{equiv:a}.

Similarly,  the second condition in \eqref{equiv:b} says that $$\ip_x^!
i_Z^!(\F)=\ip_x^*i_Z^!(\F)\in  D^{\geq n}(\cO_{Z,x}\md)$$
(the equality here is due to the fact $x$ is generic in $Z$).
We rewrite
$$\ip_x^!i_Z^!(\F)
=R\Hom_{\cO_x}( \cO_{Z,x}, \ip_x^!(\F)),$$ and see that 
 the lowest cohomology of  $\ip_x^!(\F)$ and of 
$R\Hom_{\cO_x}( \cO_{Z,x},\ip_x^!(\F))$ occur in the same degree. Indeed, the functor
$\Hom_{\cO_x}( \cO_{Z,x},\underline{\ \ })$ is left exact, and kills no 
torsion modules, while cohomology objects of $\ip_x^!(\F)$ are torsion $\cO_x$-modules.
\end{proof}

\begin{Lem}[cf. e.g.
\cite{H}, Theorem V.4.1]\label{lim}
Let $\ip:\bfZ \imbed X\tp$ be the embedding of a closed
subspace. 
For any $\F \in D^-_{coh}(X)$, $\G\in D^+_{qcoh}(X)$ we have
$$                                 
\Hom(\F ,\ip_*\ip^!(\G))=\varinjlim_Z \Hom(\F ,i_{Z*}i_Z^!(\G)),
$$                                 
where $Z$ runs over the set of closed 
subschemes of $X$ with the underlying
topological space $\bfZ$.
\end{Lem}

\begin{proof}
Let us represent $\G$ by a bounded below complex $I_\G$ of injective quasi-coherent sheaves.
Then $i_{Z*}i_Z^!(I_\G)$ represents $i_{Z*}i_Z^!(\G)$, and $\ip_*\ip^!(I_\G)$ represents 
$\ip_*\ip^!(\G)$. 

However, $i_{Z*}i_Z^!(I_\G)$ is not a complex of injective modules, so it is not suitable for
computing $\Hom(\F,\ip_*\ip^!(\G))$. We therefore introduce a resolution of $\F$. Note that for
an open affine embedding $j_U:U\hookrightarrow X$, the extension by zero $(j_U)_!\cO_U$ is
`projective with respect to $\qcoh(X)$' in the sense that the functor $\Hom((j_U)_!\cO_U,\underline{\ \ })$ is exact on $\qcoh(X)$. Let us represent $\F$ by a bounded
above complex $P_\F$ whose terms are finite direct sums of sheaves
of the form $(j_U)_!\cO_U$ for open affine embeddings $j_U:U\hookrightarrow X$. The construction of such complex $P_\F$ is sketched in Remark~\ref{PF}.

Now $R\Hom(\F, i_{Z*}i_Z^!(\G))$ can be computed
by the complex $$\Hom^\bullet(P_\F, i_{Z*}i_Z^!(I_\G)).$$
On the other hand, 
$R\Hom(\F ,\ip_*\ip^!(\G))$ is computed by the complex 
$$\Hom^\bullet(P_\F, \ip_* \ip^!(I_\G)).$$ 
Finally, 
$$\ip_*\ip^!(I_\G)=\cupl_Z i_{Z*}i_Z^!(I_\G),$$ and therefore 
$$\Hom^\bullet
(P_\F, \ip_*\ip^!(I_\G))=\cupl_Z \Hom^\bullet (P_\F ,i_{Z*}i_Z^!(I_\G)).$$
This implies the lemma. 
\end{proof}

\begin{Rem}\label{PF} Let us sketch a construction of $P_\F$. By \cite[Corollary~5.5]{BN},
$\F$ can be represented by a bounded above complex of quasi-coherent $\cO_X$-modules $\F^\bu$ 
(which can be chosen to be coherent by Corollary~\ref{odinhren}). 
We can now proceed in one of the two ways.

On the one hand, for a Noetherian scheme $X$, the sheaf $\cO_X$ is Noetherian in the category of sheaves of $\cO_X$-modules 
(and not just quasi-coherent $\cO_X$-modules); see the proof of \cite[Theorem~II.7.8]{H}. 
It follows that $j_!\cO_U$ is Noetherian for any open embedding $j:U\imbed X$, 
and that a sheaf of $\cO_X$-modules is Noetherian if and only if it is a quotient of a finite direct sum of sheaves of the form $j_!\cO_U$; here we can restrict
ourselves to affine open embeddings $j:U\imbed X$. Also, any coherent $\cO_X$-module is 
Noetherian. Finally, for any bounded above complex of $\cO_X$-modules with Noetherian cohomology
$\F^\bu$, we can construct a quasi-isomorphism $P_\F^\bu\to\F^\bu$, where $P_\F^\bu$ is of the
kind we consider; the argument is standard and similar to that used in Corollary~\ref{odinhren}.

On the other hand, here is a more explicit construction. Choose an affine hypercover $j_\bu:X_\bu\to X$ in the Zariski topology. If $X$
is semi-separated, such $X_\bu$ can be constructed by taking a finite affine open cover
$X=\bigcup U_i$, and setting
$$X_n=\coprod_{i_0<\dots<i_n} U_{i_0}\cap\dots\cap U_{i_n}.$$
(If $X$ is not semi-separated, intersections $U_{i_0}\cap U_{i_1}$ may fail to be affine, so we
must take their affine covers, and proceed recursively.) We then have, for every sheaf $\G$ on
$X\tp$, the \emph{co-\v Cech resolution} 
$$\dots\to (j_2)_!j_2^*\G\to (j_1)_!j_1^*\G\to(j_0)_!j_0^*\G$$
of $\G$. 

Choose a resolution $\F_n^\bu$ of 
$j_n^*\F$ by free $\cO_{X_n}$-modules of finite 
rank for each $n$. 
Such choices can be done in a compatible way, so that $\{\F_n^\bu\}_n$ is a 
complex of quasi-coherent sheaves on the simplicial scheme $X_\bu$. Now set $P_\F$ to
be the complex associated with a double complex
$$\dots\to (j_2)_!\F_2^\bu\to (j_1)_!\F_1^\bu\to(j_0)_!\F_0^\bu.$$ 
\end{Rem}

\section{Perverse coherent sheaves}\label{s1}
\subsection{Perverse $t$-structure on schemes}
Let us first consider the case of schemes. Suppose $X$ is a Noetherian scheme that admits a dualizing complex.

Fix a dualizing complex 
$$\DC_X\in D^b_{coh}(X).$$ The choice determines
the \emph{codimension function} $d$ on points $x\in X\tp$ such that
$\ip_x^!(\DC_X)$ is concentrated in 
homological degree $d(x)$  (see \cite{H}, \S V.7). Clearly, $d(x)$ is bounded, because $\DC_X$ has finite injective dimension. We set $\dim(x)=-d(x)$. 

\begin{Ex} Suppose $X$ is of finite type over a field. Then we can choose $\DC_X$ so that 
$\dim(x)$ equals to the dimension of the closure of $x$.
\label{finitetypescheme}
\end{Ex}

Let  $p$ (the \emph{perversity}) be an integer-valued function $p:X\tp\to\Zet$. The \emph{dual perversity} (for given $\DC_X$) is given by $\pbar(x )=-\dim(x)-p(x)$.

\begin{Def} Define  $\Dpm\subset D^-_{coh}(X)$,  $\Dpp \subset D^+_{coh}(X)$
as follows:

 $\F \in \Dpp$ if 
for any $x \in \Xtop$  we have $\ip_x ^!(\F )\in D^{\geq p(x )}(\cO _x\md)$.

$\F \in  \Dpm$ if 
for any $x \in \Xtop$  we have  $\ip_x ^*(\F )\in D^{\leq p(x )}(\cO _x\md)$.
\end{Def}

\begin{Lem}\label{predvar} 
\begin{enumerate}
\item\label{predvar:a} $\D(\Dpm)=D^{\pbar,\geq 0}(X)$, where $\D=R\uHom(\underline{\ \ },\DC_X)$ is the duality functor.

\item\label{predvar:b} Let  $i_Z:Z\imbed X$ be  a locally closed 
subscheme. Define the \emph{induced perversity} on $Z$ by $p_Z=p\circ i_Z:Z\tp\to \Zet$.
Then 
$$i_Z^*(\Dpm)\subset D^{p_Z,\leq 0}(Z)\qquad\text{and}\qquad i_Z^!(\Dpp)\subset D^{p_Z,\ge 0}(Z).$$

\item\label{predvar:c} In the situation of \eqref{predvar:b} assume that $Z$ is closed. Then 
$$i_{Z*}(D^{p_Z,\leq 0}(Z))\subset\Dpm\qquad\text{and}\qquad i_{Z*}(D^{p_Z,\geq 0}(Z))\subset\Dpp.$$ 
\end{enumerate}
\end{Lem}

\begin{proof} \eqref{predvar:a} One knows from \cite{H}, Section~V.6 that 
$$\ip_x^!(\D(\F))
= \Hom_{\cO_x}(\ip_x^*(\F), I_{\cO_x})[-\dim(x)]\quad (\F\in D^b_{coh}(X),x\in\Xtop),$$ 
where $I_{\cO_x}$ is the injective hull of the residue field of $\cO_x$. 
Since $\Hom_{\cO_x}
(\underline{\ \ }, I_{\cO_x})$ is exact and kills no finitely generated 
$\cO_x$-module, \eqref{predvar:a} follows.

\eqref{predvar:b} follows from Lemma \ref{equiv}; in view of this lemma
if $\F\in D^{p,\leq 0}$, then for any $x \in \Ztop\subset \Xtop$
there exists a subscheme $Z'\subset Z$ with generic point $x $ such 
that $i_{Z'}^*(\F)=i_{Z'}^*(i_Z^*(\F))\in D^{\leq p(x )}(Z)$. This
implies $i_Z^*(\F)\in D^{p_Z,\leq 0}(Z)$. The argument for $i_Z^!(\F)$ is similar.

\eqref{predvar:c} is obvious. 
\end{proof}

\begin{Rem*} Set $\cP_{X,p}=\Dpp\cap\Dpm$. It follows from Theorem~\ref{t_str} that $\cP_{X,p}$ is
an abelian category (the category of perverse coherent sheaves). Lemma~\ref{predvar}\eqref{predvar:a} implies that the duality $\D$ yields an anti-equivalence between $\cP_{X,p}$ and
$\cP_{X,\pbar}$.

In particular, suppose $p=0$. Then $\cP_{X,p}=\coh(X)$, while $\cP_{X,\pbar}$ is the category
of Cohen-Macaulay complexes. See \cite[Section~6]{YZ} for the definition of Cohen-Macaulay
complexes; the claim that $\D(\cP_{X,\pbar})=\coh(X)$ is contained in \cite[Theorem~6.2]{YZ}.
\end{Rem*}

It is clear that the categories $\Dpm$, $\Dpp$ are local with respect to flat Gorenstein 
coverings. 

\begin{Lem}\label{smoothloc} Suppose $\pi:Y\to X$ is a faithfully flat Gorenstein morphism of finite type. Consider on $Y$ the perversity
$\pi^*p=p\circ \pi$. Then for $\F\in D_{coh}(X)$,
\begin{align*}
\F\in D^{p,\ge 0}(X)&\text{ if and only if }\pi^*\F\in D^{\pi^*p,\ge 0}(Y)\\
\F\in D^{p,\le 0}(X)&\text{ if and only if }\pi^*\F\in D^{\pi^*p,\le 0}(Y).
\end{align*}
\qed
\end{Lem}

\begin{Rem*} 
One could consider on $Y$ the perversity $\pi^!p$ given by $$\pi^!p(y)=p(\pi(y))+
\codim(\overline{\{y\}}\subset\overline{\pi^{-1}(\pi(y))})\quad(y\in Y\tp).$$
Then $$\pi^!\overline p=\overline{\pi^* p},$$
assuming the dualizing complex on $Y$ is the pull-back of the dualizing complex on $X$. It is easy to
see that Lemma~\ref{smoothloc} remains true for the perversity $\pi^!p$; this statement is dual to Lemma~\ref{smoothloc} in the sense of the Grothendieck-Serre duality.
\end{Rem*}

\begin{Prop}\label{Hom0sch} For $\F\in \Dpm$, $\G\in D^{p,>0}(X)$ we have
  $\Hom(\F , \G)=0$. 
\end{Prop}

\begin{proof} Fix $\F$ and $\G$.
We proceed by Noetherian induction in $X$; thus we can assume
that the statement with $(X,p)$ replaced by $(Z,p_Z)$ for all 
closed subschemes $Z\subsetneq X$ is known.

Let $x $ be a generic point of  $X$.
By Lemma~\ref{equiv}, there exists an open subscheme $j:U\imbed X$ containing $x$
 such that $j^*(\F )\in D^{\leq
p(x )}_{coh}(U)$ and $j^*(\G)=j^!(\G)\in D^{>p(x )}_{coh}(U)$.
 Thus, of course, $\Hom(j^*(\F ),
j^*(\G))=0$. 

Let $\ip$ denote the closed embedding of $X\tp-U\tp$ into $X\tp$. 
It induces a distinguished triangle 
$$\ip_*\ip^! \G \to \G \to j_*j^*(\G)\to \ip_*\ip^! \G[1]$$
in $D^b_{qcoh}(X)$. 
By Lemma~\ref{lim},  
$$\Hom (\F ,\ip_*\ip^! \G)=
\varinjlim_Z \Hom(\F ,i_{Z*}i_Z^!(\G))=\varinjlim_Z
\Hom(i^*_Z(\F ),i_Z^!(\G)),$$
where $Z$ runs over closed subschemes of $X$ whose underlying set is $X\tp-U\tp$.
Notice finally that for any such $Z$,
$$i^*_Z(\F ) \in D^{p_Z, \leq 0}(Z)\qquad\text{and}\qquad i_Z^!(\G) \in  D^{p_Z, > 0}(Z)$$ 
by Lemma~\ref{predvar}\eqref{predvar:b}, so  
$$\Hom(i^*_Z(\F ),i_Z^!(\G))=0$$ by the induction hypotheses.
This implies the desired equality $\Hom(\F ,\G)=0$, since
$\Hom(\F ,j_*j^*(\G))=\Hom(j^*(\F ),j^*(\G))=0$. 
\end{proof}

If perversity $p$ is both monotone and comonotone (Definition~\ref{mon}), the categories $\Dpp$ and $\Dpm$ yield a $t$-structure on $D^b_{coh}(X)$.
To avoid repeating the argument, we prove this for perversities on stacks in Theorem~\ref{t_str}.

\subsection{Perverse $t$-structure on stacks}
Now let $\sX$ be a Noetherian semi-separated stack that admits
a dualizing complex. Fix a dualizing complex $\DC_\sX\in D^b_{coh}(\sX)$.
Recall that $\sX\tp$ is set of points of $\sX$ equipped with the Zariski topology. 
By definition, a \emph{perversity} on $\sX$ is a function $p:\sX\tp\to\Zet$. 

Fix a presentation $\pi:X\to\sX$ of $\sX$. 
On $X$, we consider the induced perversity $\pi^*p=p\circ\pi:X\tp\to\Zet$
and the induced dualizing sheaf $\DC_X=\pi^*\DC_\sX$.
For a point $x\in\sX\tp$, take $y\in\pi^{-1}(x)\in X\tp$, and define the codimension function $$d(x)=d(y)-\codim(y/x),$$ where
$\codim(y/x)$ is the codimension of the closure of $y$ in the fiber of $\pi$ over $x$.
The definition is independent of presentation $\pi$ and preimage $y$. Set $\dim(x)=-d(x)$,
and define the dual perversity $\pbar:\sX\tp\to\Zet$
by the same formula as in the case of schemes.

\begin{Ex} Suppose $\sX$ is of finite type over a field. Then $\DC_\sX$ can be chosen in such a way
that $\dim(x)$ is the dimension of $\overline{\{x\}}$ (in the sense of stacks).
\end{Ex}

\begin{Def}\label{deft} 
\begin{align*}
D^{p,\le 0}(\sX)&=\{\F\in D^-_{coh}(\sX):\pi^*\F\in D^{\pi^*p,\le 0}(X)\}\\
D^{p,\ge 0}(\sX)&=\{\F\in D^+_{coh}(\sX):\pi^*\F\in D^{\pi^*p,\ge 0}(X)\}.
\end{align*}
\end{Def}

It follows from Lemma~\ref{smoothloc} that Definition~\ref{deft} does not depend on the choice of presentation $\pi:X\to\sX$. It is clear that Lemmas~\ref{predvar} and \ref{smoothloc}
hold for stacks. Let us now prove Proposition~\ref{Hom0sch} in these settings.

\begin{Prop}\label{Hom0} For $\F\in D^{p,\le0}(\sX)$, $\G\in D^{p,>0}(\sX)$ we have
  $\Hom(\F , \G)=0$. 
\end{Prop}

\begin{proof}
Fix $\F$ and $\G$, and let us prove that $\Ext^i(\F,\G)=0$ for $i\le 0$ by induction. The statement holds for $i\ll 0$,
because $\F\in D^-(\sX)$ and $\G\in D^+(\sX)$. It remains to prove that if the statement holds for all $i<i_0$, it also holds 
for $i=i_0$, where $i_0\le 0$ is fixed. Now \cite[Proposition~3.2.2]{BBD} shows that
homomorphisms from $\F$ to $\G[i_0]$ form a sheaf in the flat topology on $\sX$. To complete the proof, note that $\Hom(\pi^*\F,\pi^*\G[i_0])=0$ by Proposition~\ref{Hom0sch}.
\end{proof}

\begin{Def}\label{mon}
 A perversity function $p$ is 

 {\it monotone} if $p(x')\geq p(x)$ whenever $x'\in\overline{\{x\}}$, $x,x'\in \sX^{top}$;

 {\it strictly monotone} if $p(x')> p(x)$ whenever $x'\in\overline{\{x\}}$, $x,x'\in \sX^{top}$, $x\ne x'$;

 (strictly) {\it comonotone} if the dual perversity $\pbar(x )=-\dim(x )-p(x )$ is (strictly) monotone.
\end{Def}

\begin{Thm}\label{t_str} Suppose that a perversity $p$ is monotone
and comonotone. Then 
$(D^{p,\le0}(\sX)\cap D^b(\sX) , D^{p,\ge0}(\sX)\cap D^b(\sX))$ define a $t$-structure on $D^b_{coh}(\sX)$.
\end{Thm}

\begin{proof}
Set $D^{p,\le0}(\sX)^b=D^{p,\le0}(\sX)\cap D^b(\sX)$, $D^{p,>0}(\sX)^b=D^{p,>0}(\sX)\cap D^b(\sX)$.
In view of Proposition~\ref{Hom0}, we have only to construct for
any $\F\in D^b_{coh}(\sX)$ a distinguished triangle
$$\F'\to \F\to \F''\to\F'[1]$$ 
with $\F'\in D^{p,\leq 0}(\sX)^b$, $\F''\in D^{p,>0}(\sX)^b$.

We abuse notation as follows. For a category
$\A$, we denote the set of isomorphism classes of $\ob(\A)$ by the same letter 
$\A$. In particular, for a functor $F:\A\to \B$, we denote by $F(\A)$ the image of
the induced map from the set of isomorphism classes of $\ob(\A)$ to that of $\ob(\B)$.

Let us also adopt the following convention
(see \cite{BBD}, 1.3.9).  If $D'$, $D''$ are sets of (isomorphism classes)
of objects of a triangulated category $D$, then $D'*D''$ is the set
of  (isomorphism classes) of objects of $D$, defined by the following condition:
$B\in D'*D''$ if and only if there exists a distinguished triangle 
$$A\to B\to C\to A[1]\quad(A\in D',C\in D'').$$
The octahedron axiom implies (see \cite{BBD}, Lemma 1.3.10)
that the $*$ operation is associative: $$(D'*D'')*D'''
=D'*(D''*D''').$$ Thus the meaning of the
expression $D_1*\dots *D_n$ is unambiguous. 

In this notation, we need to prove that 
\begin{equation*}
D^b_{coh}(\sX)\subset D^{p,\leq 0}(\sX)^b*D^{p,>0}(\sX)^b.
\end{equation*}

We proceed by Noetherian induction; thus we can assume 
that for any closed substack $\sZ\subsetneq\sX$,
\begin{equation*}
D^b_{coh}(\sZ)\subset D^{p_{\sZ},\leq 0}(\sZ)^b*D^{p_{\sZ},>0}(\sZ)^b,
\end{equation*}
where $p_\sZ:\sZ\tp\to\Zet$ is the induced perversity.
By Lemma~\ref{predvar}\eqref{predvar:c}, it suffices to show that 
\begin{equation}\label{reduce}
D^b_{coh}(\sX)\subset\cupl_\sZ D^{p,\leq 0}(\sX)^b* i_{\sZ*}D^b_{coh}(\sZ) * D^{p,>0}(\sX)^b,
\end{equation}
where the union is over closed substacks $i_{\sZ}:\sZ\hookrightarrow\sX$, $\sZ\ne\sX$.

Let us prove \eqref{reduce}. Fix $\F\in D^b_{coh}(X)$. Choose a generic point $y\in\sX^{top}$
and a closed irreducible substack $i_\sY:\sY\hookrightarrow\sX$ that coincides with $\sX$ in
a neighborhood of $y$. Consider $i_{\sY*}i^!_\sY(\F)\in D^+_{coh}(\sX)$, and set
$$\F^-=\tau_{\leq p(y)}^{stand}(i_{\sY*}i^!_\sY(\F)).$$
Here $\tau_{\le n}^{stand}:D(\sX)\to D^{\le n}(\sX)$ is the truncation functor
for the usual $t$-structure ($n\in\Zet$).
Clearly $\F^-\in D^b_{coh}(\sX)$.
Moreover, $\F^-\in D^{p,\leq 0}(\sX)^b$, because it is supported on $\sY$
and $p$ is monotone. Note that $\F^-$ is equipped with
a canonical morphism $\F^-\to \F$. Set
$$\F_1=\cone(\F^-\to\F).$$ 
Over an open subset of $\sY$,
$\F_1$ is concentrated in cohomological degrees above $p(y)$.

The dual procedure (in the sense of Grothendieck-Serre duality)
gives $\F^+ \in D^{p,>0}(\sX)$ and a morphism $f:\F_1\to \F^+$ that is an isomorphism over $y$.
More precisely, set
$$\F^+= \D( \tau^{stand}_{< \pbar(y ) } i_{\sY *}i^!_\sY(\D(\F_1))).$$
Since $p$ is comonotone, we see by Lemma \ref{predvar}\eqref{predvar:a} that
$\F^+\in D^{p,>0}(\sX)^b$. Note that the duality is exact over the generic point $y$
(\cite{H}, \S V.6), which implies that $f$ is an isomorphism at $y$.

Thus $$\F^0=\cone(\F_1\to \F^+)$$ vanishes at $y$.
By Lemma~\ref{closedimbed}, there is a closed substack $i_\sZ:\sZ\hookrightarrow\sX$, $\sZ\ne\sX$ such that
$$\F^0\cong i_{\sZ*}(\F_\sZ)\quad(\F_\sZ\in D^b_{coh}(\sZ)).$$

Therefore, $$\F\in \{\F^-\}*\{\F^0[-1]\}
*\{\F^+\}\subset D^{p,\leq 0}(\sX)^b*i_{\sZ*}D^b_{coh}(\sZ) * D^{p,>0}(\sX)^b,$$
 which proves \eqref{reduce}. 
\end{proof}

\begin{Rem*} Construction of an object $\F^+\in D^{p,>0}(\sX)^b$
 with given generic fiber (and with a morphism from a given object)
is the only place where we use the duality formalism on stacks.
\end{Rem*}

\begin{Cor} $(D^{p,\le0}(\sX), D^{p,\ge0}(\sX))$ define a $t$-structure on $D_{coh}(\sX)$.
In other words, Theorem~\ref{t_str} holds for the unbounded derived category.
\end{Cor}
\begin{proof} We need to prove that
\begin{equation*}
D_{coh}(\sX)=D^{p,\leq 0}(\sX)*D^{p,>0}(\sX),
\end{equation*}
where $*$ is defined in the proof of Theorem~\ref{t_str}. Clearly,
$$D^{p,\leq0}(\sX)\supset D^{<-N}_{coh}(\sX)\quad\text{and}\quad
D^{p,>0}(\sX)\supset D^{>N}_{coh}(\sX)\quad(N\gg0).
$$
Therefore,
\begin{multline*}
D^{p,\leq 0}(\sX)*D^{p,>0}(\sX)=
D^{<-N}_{coh}(\sX)*D^{p,\leq 0}(\sX)*D^{p,>0}(\sX)*D^{>N}_{coh}(\sX)\supset\\
D^{<-N}_{coh}(\sX)*D^b_{coh}(\sX)*D^{>N}_{coh}(\sX)=D_{coh}(\sX).
\end{multline*}
\end{proof}

\begin{Cor}\label{cohext}
Let $j:\sU\imbed\sX$ be an open substack, $p:\sX\tp\to \Zet$ be a monotone and
comonotone perversity, and $\F \in D^{p,\geq 0}(\sU)\cap D^b_{coh}(\sU)$.
Consider $j_*(\F)\in D^b_{qcoh}(\sX)$, and let $n=\min\limits_{x\not\in \sU\tp}
p(x)$. Then $\tau^{stand}_{\leq n-2}(j_*(\F))$ has coherent cohomology.
Recall that $\tau^{stand}_{\leq n-2}$ is the truncation functor for the usual $t$-structure.
\end{Cor}

\begin{proof} Passing to a presentation of $\sX$, we can assume that $X=\sX$ is a scheme and $U=\sU$ is an open subscheme of $X$. 
Let $\Ftil\in D^b_{coh}(X)$ be any extension of $\F$
(see Lemma~\ref{exte}). Truncating $\Ftil$ with respect to the perverse t-structure,
we can achieve that $\Ftil\in D^{p,\geq 0}(X)$. Consider the closed embedding
$\ip:X\tp-U\tp\hookrightarrow X\tp$. It yields a distinguished triangle
$$\ip_*\ip^!(\Ftil) \to \Ftil \to j_*(\F)\to \ip_*\ip^!(\Ftil)[1].$$ 
Since $\Ftil$ has coherent cohomology, it is enough to show that
$$\tau^{stand}_{\leq n-2}(\ip_*\ip^!(\Ftil)[1])\in D^b_{coh}(X).$$
However, the assumption $\Ftil\in D^{p,\geq 0}(X)$ implies
that for any closed subscheme $Z\subset X$ with $Z\tp=X\tp-U\tp$, we have
$$i_Z^!(\Ftil)\in D^{p,\geq 0}(Z)\subset D^{\geq n}_{coh}(Z).$$
Hence $$\ip_*\ip^!(\Ftil)\in D^{\geq n}_{qcoh}(X),$$ and
 $\tau^{stand}_{\leq n-2}(\ip_*\ip^!(\Ftil)[1])=0$. 
 \end{proof}

\begin{Rem}\label{Grfin}
This statement and idea of proof of Corollary~\ref{cohext} are copied
from Deligne's message to the second author. 
(Possible mistakes belong to the authors.) It was pointed out to us by 
Deligne that Corollary~\ref{cohext}
 is equivalent to
the Grothendieck Finiteness Theorem, \cite[VIII.2.1]{SGA2}.
\end{Rem}

\section{Coherent IC sheaves}\label{IC} 
\subsection{IC extension}
Suppose now that $p:\sX\tp\to\Zet$ is a monotone and comonotone perversity
function. As before, $\sX$ is a semi-separated Noetherian stack admitting a dualizing sheaf.
Denote by $\cP=\cP_\sX=\cP_{\sX,p}\subset D^b_{coh}(\sX)$ the core of the $t$-structure 
constructed in the previous section.

Let  $\bfZ\subset\sX\tp$ be a closed subset. Define auxiliary perversities 
$p^\pm=p^\pm_{(\bfZ)}:\sX\tp\to\Zet$ by
$$p^-(x)=\begin{cases}p(x), &x\not\in\bfZ\\p(x)-1,&x\in\bfZ;\end{cases}\qquad
p^+(x)=\begin{cases}p(x), &x\not\in\bfZ\\p(x)+1,&x\in\bfZ.\end{cases}$$

\begin{Lem}\label{orto}
 Let $\F\in\cP_\sX$.
\begin{enumerate}
\item\label{orto:a} The following conditions are equivalent:
\begin{enumerate}
\item\label{orto:i} $\F\in D^{p^-,\leq 0}(\sX)$.
\item\label{orto:ii} $i_\sZ^*(\F)\in D^{p_{\sZ},<0}(\sZ)$ for any closed substack
$\sZ\subset\sX$ with $\sZ\tp\subset\bfZ$.
\item\label{orto:iii} $\Hom(\F, \G)=0$ for all $\G\in \cP$ such that $\supp\G\subset\bfZ$.
\end{enumerate}
\item\label{orto:b} The following conditions are equivalent:
\begin{enumerate}
\item $\F\in D^{p^+,\ge 0}(\sX)$.
\item $i_\sZ^!(\F)\in D^{p_{\sZ},>0}(\sZ)$ for any closed substack $\sZ\subset\sX$ with $\sZ\tp\subset\bfZ$.
\item $\Hom(\G, \F)=0$ for all $\G\in \cP$ such that $\supp\G\subset\bfZ$.
\end{enumerate}
\end{enumerate}
\end{Lem}

\begin{proof}  We prove only \eqref{orto:a}; the proof of \eqref{orto:b} is completely analogous.  
The equivalence between \eqref{orto:i} and \eqref{orto:ii} follows from Lemma~\ref{equiv}.
Let us show \eqref{orto:ii}$\iff$\eqref{orto:iii}. 

Let $\sZ$ be a closed substack with $\sZ\tp\subset\bfZ$. Since $\F\in\cP_\sX$,  $$i_{\sZ}^*(\F)\in D^{p_{\sZ},\leq 0}(\sZ).$$
On the other hand, given any triangulated category $D$ with a $t$-structure and an object $A\in
D^{\leq 0}$, it is clear that $A\in D^{<0}$ if and only if $\Hom(A,B)=0$ for all $B$ in the core of the
 $t$-structure. Therefore, $i_{\sZ}^*(\F)\in D^{p_{\sZ},<0}(\sZ)$ if and only if
 $$\Hom(\F, i_{\sZ*}(\G))=\Hom(i_{\sZ}^*(\F),\G)=0\quad\text{for all }\G\in\cP_\sZ.$$
 Now the statement follows from Lemma~\ref{closedimbed}.
 \end{proof}

Suppose now that $j:\sU\imbed\sX$ is an open substack, and consider the 
perversities $p^\pm=p^\pm_{(\sX-\sU)}$. Define a full subcategory
$\cP_{!*}(\sU)\subset \cP_{\sX}$ by
  $$\cP_{!*}(\sU)= D^{p^-,\leq 0}(\sX) \cap  D^{p^+,\geq 0}(\sX).$$

\begin{Thm}\label{GM}
 Suppose that $p(x')>p(x )$,
  $\pbar (x ')>\pbar (x )$ for any $x \in \sU\tp$, 
$x '\in\overline{\{x\}}-\sU\tp$. 
Then $j^*$ induces an equivalence between $\cP_{!*}(\sU)$ and $\cP_\sU$.

The inverse equivalence is denoted by $$j_{!*}:\cP_\sU\to \cP_{!*}(\sU)\subset \cP_{\overline\sU},$$ 
and is called the functor of \emph{minimal} (or \emph{Goresky-MacPherson},
or \emph{IC}) extension.
\end{Thm}
\begin{proof} By the hypotheses, both $p^-$ and $p^+$ are monotone and comonotone perversities on $\sX\tp$.
Hence they define  $t$-structures on $D^b_{coh}(\sX)$;
let $\tau^-$,
$\tau^+$ be the corresponding truncation functors.

Define a functor $J_{!*}$ from $D^b_{coh}(\sX)$ to itself by
$$J_{!*}=\tau^-_{\leq 0} \circ \tau^+_{\geq 0}.$$
We need the following claim.

\begin{Lem}\label{JGM}
\begin{enumerate}
\item\label{JGM:a} $J_{!*}$ takes values in $\cP_{!*}(\sU)$.
\item\label{JGM:b} If a morphism $f:\F\to \G$ in $D^b_{coh}(\sX)$ is such that 
$f|_\sU$ is an isomorphism, then $J_{!*}(f)$ is an isomorphism.
\end{enumerate}
\end{Lem}

\begin{proof} Fix $\F\in D^b_{coh}(\sX)$. By construction, $J_{!*}(\F)\in D^{p^-,\leq 0}(\sX)$. 
Set $$\F_1=\tau^+_{\geq 0}\F.$$ Then $J_{!*}(\F)$ fits in a distinguished triangle
$$(\tau^-_{>0}\F_1) [-1] \to J_{!*}(\F) \to \F_1\to \tau^-_{>0}\F_1.$$ 
Since $\F_1\in D^{p^+,\geq 0}(\sX)$ and $(\tau^-_{>0}\F_1) [-1]\in D^{p^-, \geq 2}(\sX)
\subset D^{p^+, \geq 0}(\sX)$, part~\eqref{JGM:a} follows.

Let $f$ be as in \eqref{JGM:b}. Then $J_{!*}(f)$ is a morphism in $\cP_{!*}(\sU)$
that is an isomorphism over $\sU$. Therefore, its kernel and cokernel are
objects of $\cP_{\sX}$ supported by $\sX-\sU$.
But by Lemma~\ref{orto}, $J_{!*}(\F)$ and $J_{!*}(\G)$ have neither subobjects 
nor quotients supported by $\sX -\sU$.
\end{proof}

We now finish the proof of Theorem~\ref{GM}.
By Lemmas~\ref{exte} and \ref{JGM}, the functor
$J_{!*}$ decomposes as 
$$ D^b_{coh}(\sX)\overset{j^*}\to D^b_{coh}(\sU)\to\cP_{!*}(\sU)$$
for a canonically defined functor 
$\jtil_{!*}:D^b_{coh}(\sU)\to \cP_{!*}(\sU)$.
Set $j_{!*}=\jtil_{!*}|_{\cP_\sU}$. By construction, $j^*\circ j_{!*} =id_{\cP_\sU}$
canonically. Also $j_{!*}\circ j^*|_{\cP_{!*}(\sU)} =id$ canonically, because
$J_{!*}|_{\cP_{!*}(\sU)} =id$. Thus $j^*$ and $j_{!*}$ are  inverse
equivalences between $\cP_{!*}(U)$ and $\cP_U$. 
\end{proof}

\begin{Rem*} Note that the hypotheses of Theorem~\ref{GM} are always satisfied if $p$ is strictly
monotone and strictly comonotone.
\end{Rem*} 

We need some elementary properties of the minimal extension. Let us keep the hypotheses of Theorem~\ref{GM}.

\begin{Lem} For any $\F\in\cP_\sX$, $j_{!*}(\F|_\sU)$ is a subquotient of $\F$ in the abelian category $\cP_\sX$.
\end{Lem}
\begin{proof} Since $\F\in\cP_\sX$, we have
\begin{align*}
\tau^+_{\ge 0}\F&\in D^{p^+,\ge 0}(\sX)\subset D^{p,\ge 0}(\sX)\\
\tau^+_{<0}\F&\in D^{p^+,<0}(\sX)\subset D^{p,\le 0}(\sX).
\end{align*}
Now the distinguished triangle
$$\tau^+_{<0}\F\to\F\to\tau^+_{\ge0}\F\to\tau^+_{<0}\F[1]$$
implies that $\tau^+_{\ge 0}(\F),\tau^+_{<0}(\F)\in\cP_\sX$. In other words, $\tau^+_{\ge 0}(\F)$ is a quotient of $\F$.

Repeating the argument, we see that $$J_{!*}(\F)=\tau^-_{\le 0}\tau^+_{\ge 0}\F=j_{!*}(\F|_\sU)\in\cP_\sX$$
is a subobject of $\tau^+_{\ge 0}(\F)$, as required.
\end{proof}

\begin{Lem} \label{closedIC} Let $\overline i:\sY\hookrightarrow\sX$ be a closed substack. Set $\sV=\sU\cap\sY$. Then
$$j_{!*}\circ i_*\simeq \overline i_*\circ j'_{!*}.$$
Here $i$ is the closed embedding $\sV\hookrightarrow\sU$, and $j'$ is the open embedding $\sV\hookrightarrow\sY$.
\end{Lem}
\begin{proof} Indeed, $\overline i_*(\cP_{!*}(\sV))\subset\cP_{!*}(\sU)$ by
Lemma~\ref{predvar}\eqref{predvar:c}. 
\end{proof}

\begin{Def} Let $j:\sU\hookrightarrow\sX$ be a locally closed embedding that satisfies the hypotheses
of Theorem~\ref{GM}:  
$$p(x')>p(x )\quad\text{and}\quad \pbar (x ')>\pbar (x )\quad (x \in \sU\tp,x '\in\overline{\{x\}}-\sU\tp).$$ The \emph{minimal (or Goresky-MacPherson, or IC) extension
functor} $$j_{!*}:\cP_\sU\to\cP_\sX$$
is given by $i_*\circ j'_{!*}$, where $i:\sZ\hookrightarrow\sX$ is a closed substack such that $j':\sU\hookrightarrow\sZ$ is an open
embedding. By Lemma~\ref{closedIC}, the functor is independent of the choice of $\sZ$.
\end{Def}

\subsection{Irreducible perverse sheaves}
For the rest of the paper, we assume that $\sX$ is a semi-separated
stack of finite type over a field $\kk$. 

Let $x$ be a point of $\sX$. Let us represent it by a morphism $\xi:\spec K\to\sX$ for a field extension
$K\supset\kk$. Denote by $G_\xi$ the automorphism group of $\xi$; it is a $K$-group scheme of finite type.  
Note that $\dim(G_\xi)$ depends only on $x\in\sX\tp$, and that it is an upper semi-continuous function of
$x$ (because $G_\xi$ can be thought of as the fiber of the diagonal morphism $\sX\to\sX\times\sX$ over
$(\xi,\xi)$). 

\begin{Ex} Suppose $\sX=X/G$, where $X$ is a semi-separated scheme of finite type over $\kk$ and 
$G$ is an affine group scheme of finite type acting on $X$. 
A point $\xi:\spec K\to X$ defines a point of $\sX$; the corresponding automorphism group is 
the stabilizer of $\xi$. 
\end{Ex}

Denote by $\sX\tpl\subset\sX\tp$ the set of points $x\in\sX\tp$ that are defined over the
algebraic closure $\kbar\supset\kk$.
For any presentation $\pi:X\to\sX$, $x\in\sX\tp$ is defined over
$\kbar$ if and only if it can be lifted to a $\kbar$-point of $X$, because $\pi$ has finite type. 

\begin{Lem} For $x\in\sX\tp$, the set $\{x\}$ is locally closed if and only if $x\in\sX\tpl$.
\label{locallyclosed}
\end{Lem} 
\begin{proof}
Suppose $\{x\}$ is locally closed. Then there is a locally closed substack $\sY\subset\sX$ such that 
$\sY\tp=\{x\}$. However, $\sY\tpl\ne\emptyset$ 
(because its presentation has $\kbar$-points), and therefore $x$ is defined over $\kbar$.

Suppose $x$ is defined over $\kbar$. Passing to a locally closed subset of $\sX$, we may assume that 
$\dim(G_\eta)$ is constant for all $\eta:\spec K\to\sX$. But then $x\in\sX\tp$ has minimal dimension (equal to
$-\dim(G_\eta)$), and therefore 
$\{x\}\subset\sX\tp$ is closed. (For another proof, note that $\{x\}\subset\sX\tp$ is the image of a morphism
$\spec K\to\sX$, where $K\supset\kk$ is a finite extension. This morphism is of finite type, therefore $\{x\}$ is a constructible set. This implies $\{x\}$ is locally closed.)
\end{proof}

\begin{Lem} Suppose $x\in\sX\tp-\sX\tpl$. Then there is a point $x'\in\overline{\{x\}}$ of dimension $\dim(x')=\dim(x)-1$.
\label{codim1}
\end{Lem}
\begin{proof} For a presentation $\pi:X\to\sX$, consider preimage $\pi^{-1}(x)\subset X\tp$. Choose a point $y\in\pi^{-1}(x)$ that is closed in
the induced topology, and set $Y=\overline{\{y\}}\subset X$. Note that $\dim Y>0$, 
otherwise $x$ would be defined over $\kbar$. 

The restriction $\pi|_Y:Y\to\sX$ is of finite type; therefore, we can replace $Y$ with its open subset such that 
the dimension of fibers of $\pi|_Y$ is constant and equal to $\dim Y-\dim x$. (Actually, it is easy to see that
$\dim Y-\dim x=\dim G_\xi$, where $\xi$ represents $x$.) Take a point $y'\in Y$ such that $\codim\overline{\{y'\}}=1$, and set 
$x'=\pi(y')$. We claim that $\dim x'=\dim x-1$. 

Indeed, $\dim x'<\dim x$, because $x'\in\overline{\{x\}}$ and $x'\ne x$. On the other hand,
$$\dim x'=\dim\overline{\{x'\}}\ge\dim(\pi^{-1}(\overline{\{x'\}})\cap Y)-(\dim Y-\dim x)=\dim x-1.$$
\end{proof}

\begin{Rem*} Lemmas~\ref{locallyclosed} and \ref{codim1} are also easy to derive from \cite[Corollary 10.8]{La}. If $\sX$ is of the form $X/G$, Lemmas~\ref{locallyclosed} and \ref{codim1} follow from Rosenlicht's Theorem (\cite{Rosenlicht}, see also exposition in \cite[Section~2.3]{RosenlichtSum}). Essentially, Rosenlicht's Theorem allows us to construct, for an action of a 
group $G$ on an irreducible variety $X$, a rational quotient: a rational dominant map $q:X\to Q$ that identifies $\kk(Q)$
with $\kk(X)^G$. Moreover, it claims that $q$ generically separates $G$-orbits on $X$. This allows us to generically identify
$(X/G)\tp$ with $Q\tp$ and to reduce the above lemmas to corresponding statements for varieties.
\end{Rem*}

By Lemma~\ref{locallyclosed}, a point $x\in\sX\tpl$ determines a reduced locally closed substack 
$\sG_x\subset\sX$: the \emph{residue gerbe} of $x$. 
We denote the embedding $\sG_x\hookrightarrow\sX$ by $j^{(x)}$. 

\begin{Rem} \label{residue}
Let $\kk_x$ be the residue field of $x$ (see \cite[Chapter 11]{La}); since $x\in\sX\tpl$, $\kk_x\supset\kk$ is a finite extension. Suppose that $x$ is defined over $\kk_x$ so that $x$ is the image
of a morphism $$\xi:\spec\kk_x\to\sX.$$ 
Then $\xi$ provides a presentation of $\sG_x$, and 
$\sG_x$ is isomorphic to the classifying stack of $G_\xi$. In particular, quasi-coherent sheaves on $\sG_x$
identify with representations of $G_\xi$.

In general, we can always represent $x$ as the image of $\xi:\spec K\to\sX$, where $K$ is a finite extension of
$\kk_x$. Then the extension of scalars
$$\sG_x\times_{\spec\kk_x}\spec K$$
identifies with the classifying stack of $G_\xi$, so in a sense $\sG_x$ is a $\kk_x$-form of this classifying stack.
\end{Rem}

Let us fix a monotone and comonotone perversity $p:\sX\tp\to\Zet$. Recall that 
$$\cP_{\sX}\subset D^b_{coh}(\sX)$$ is the kernel of the associated t-structure.

\begin{Prop} \label{irre}
For $\F \in D^b_{coh}(\sX)$ the following statements are equivalent:

\begin{enumerate}
\item\label{irre:i} $\F$ is an irreducible object of $\cP_{\sX}$.
\item\label{irre:ii} There exists a point $x\in\sX\tpl$ such that
$p(x)<p(y)$, $\pbar(x)<\pbar(y)$ for any 
$y\in\overline{\{x\}}-\{x\}$ and an irreducible coherent sheaf $\cL$ on $\sG_x$ 
such that
$\F=j^{(x)}_{!*}(\cL[-p(x)])$.
\end{enumerate}
\end{Prop}
\begin{proof}  \eqref{irre:ii}$\Rightarrow$\eqref{irre:i}. Suppose that
$\F$ has a proper subobject $\F'$. Then $\supp(\F')\subset\overline{\{x\}}$, so by Lemma~\ref{closedimbed}, $\F'$
can be obtain by a direct image from a closed substack $\sY\subset\sX$ with generic point $x$. Replacing $\sX$
with $\sY$, we may assume that $x\in\sX\tp$ is dense, and therefore open.

Let $j:\sU\hookrightarrow\sX$ be the open substack with $\sU\tp=\{x\}$; its maximal reduced substack is $\sG_x$. 
By Lemma~\ref{closedIC},
$$\F=j_{!*}i_{*}(\cL[-p(x)]),$$
where $i:\sG_x\hookrightarrow\sU$ is the natural closed embedding. Clearly, $\cL[-p(x)]\in\cP_\sU$ is irreducible
(note that $\cP_\sU=\coh(\sU)[-p(x)]$). Now irreducibility of $\F$ follows from Lemma~\ref{orto}.

\eqref{irre:i}$\Rightarrow$\eqref{irre:ii}. Let $x$ be a generic point of $\supp(\F)$. 
Since $\F$ is irreducible, we have 
$$\Hom(\F, i_{\sZ*}(\G))=0,\qquad
\Hom( i_{\sZ*}(\G), \F)=0$$
for any closed substack $i_\sZ:\sZ\hookrightarrow\sX$ not containing $x$ and any $\G\in\cP_\sZ$. 

Thus Lemma~\ref{orto} implies that
$$i_\sZ^*(\F)\in D^{p_\sZ,<0}(\sZ),\qquad 
i_\sZ^!(\F) \in  D^{p_\sZ,> 0}(\sZ).$$
In particular, this shows that $\supp\F$ is irreducible (otherwise we can take $\sZ$ to be an
irreducible component of $\supp\F$ that does not contain $x$). 

Now take $y\in\supp\F=\overline{\{x\}}$, $y\ne x$. Then $p(x)<p(y)$:
otherwise, the coherent sheaf ${\cal H}^{p(x )}(\F)$ has a nonzero fiber at $x $,
but has  zero fiber at $y$, which contradicts the Nakayama Lemma.
A dual argument shows that $\pbar(x)<\pbar(y)$.
In particular, $\dim(y)< \dim(x)-1$. By Lemma~\ref{codim1}, $x\in\sX\tpl$. 
By Lemma~\ref{closedimbed}, $\F$ can be obtain by a direct image from a closed substack $\sY\subset\sX$ with generic point $x$. 

Consider the open substack $j:\sU\hookrightarrow\sY$ such that $\sU\tp=\{x\}$. Then Theorem~\ref{GM} implies that
$$\F=j_{!*}(\cL'[-p(x)])\quad (\cL'\in\coh(\sU)).$$ 
However, $\cL'$ has to be irreducible, and, in particular, it is supported by the reduced substack $\sG_x\subset\sU$.
\end{proof}

\begin{Ex} Oftentimes, Proposition~\ref{irre} can be made more explicit using Remark~\ref{residue}. For instance, suppose $\kk=\kbar$.
Then every point $x\in\sX\tpl$ is defined over $\kk$, and an irreducible coherent sheaf $\cL$ on $\sG_x$ is simply an irreducible representation
of $G_\xi$. Here $\xi:\spec\kk\to\sX$ is the $\kk$-point representing $x$.
\end{Ex}

\begin{Cor}\label{Art} Suppose that  
$\sX$ has finitely many points. If the perversity 
$p:\sX\tp\to\Zet$ is strictly monotone and strictly comonotone, then the  category $\cP_\sX$ is Artinian and Noetherian.
\end{Cor}

\begin{proof} In this case, every point $x\in\sX\tp$ is locally closed, so $x\in\sX\tpl$, and $j_{!*}^{(x)}$
is defined (by the hypotheses).
By induction in the number of points, one can deduce that the irreducible
objects generate the triangulated category $D^b_{coh}(\sX)$.
This implies that $\cP_\sX$ is Artinian and Noetherian.
\end{proof}

\begin{Cor} Suppose that $\sX$ has finitely many points and that the perversity $p:\sX\tp\to\Zet$ is strictly monotone and strictly
comonotone. Then classes of irreducible objects $\F\in\cP_\sX$ (described in Proposition~\ref{irre}) form a basis in $K^0(\coh(\sX))$.
\end{Cor}

\begin{proof} The derived category $D^b_{coh}(\sX)$ has two bounded t-structures: the standard t-structure and the perverse t-structure. For the
hearts of these t-structures, we get an identification
$$K^0(\coh(\sX))=K^0(\cP_\sX).$$
Now the claim follows from Corollary~\ref{Art}.
\end{proof}

\begin{Ex}\label{N}
 Let $G$ be a semisimple group over a field of characteristic 0,
or of large finite characteristic, and let $\N\subset G$ be the subvariety of
unipotent elements. Then $G$ acts on $\N$ by conjugation, and this action has
a finite number of orbits. Moreover, dimensions of orbits are known to be even.
Points $x$ of the stack $\N/G$ correspond to $G$-orbits $O\subset\N$, and 
we can define the \emph{middle perversity} $p:(\N/G)\tp$
by 
$$p(x)=-\frac{\dim(O)}{2}.$$
Obviously, $p$ is
strictly monotone and comonotone, hence by Proposition \ref{Art}
the heart of the corresponding $t$-structure is Artinian and Noetherian.
See \cite{me}, \cite{hum} for more information on this example.

\end{Ex}

\begin{Ex}
Let $G$ be as above and let $Gr$ denote the affine Grassmanian of $G$.
This is an ind-scheme acted upon by the (infinite dimensional) group
scheme $G_O$. The orbits are in bijection with the set of dominant
coweights of $G$. Moreover, it is known that orbits in a given connected
component of $Gr$ have dimension of a fixed parity. Thus the perversity
function
$p(x)= -[\frac{\dim(O)}{2}]$ is monotone and comonotone, and
 (a straightforward generalization of) Corollary \ref{Art} applies.
See \cite{BFG} for more information on this example.
\end{Ex}

\bibliographystyle{abbrv}
\bibliography{izvrat}
\end{document}